\documentclass[10pt]{amsart}
\usepackage{latexsym}
\usepackage{amsfonts}

\usepackage{amsmath,amsthm,amssymb}

\title[Intersection form, laminations and currents] {Intersection form, laminations and currents on free groups}

\author[I.~Kapovich]{Ilya Kapovich}

\address{\tt Department of Mathematics, University of Illinois at
  Urbana-Champaign, 1409 West Green Street, Urbana, IL 61801, USA
  \newline http://www.math.uiuc.edu/\~{}kapovich/} \email{\tt
  kapovich@math.uiuc.edu}

\author[M.~Lustig]{Martin Lustig}\address{\tt Math\'ematiques
  (LATP), Universit\'e Paul C\'ezanne -Aix Marseille III,\\  av. Escadrille
  Normandie-Ni\'emen, 13397 Marseille 20, France} \email{\tt Martin.Lustig@univ-cezanne.fr}

\newtheorem{thm}{Theorem}[section] \newtheorem{lem}[thm]{Lemma}
\newtheorem{cor}[thm]{Corollary} 
\newtheorem{prop}[thm]{Proposition} \theoremstyle{definition}
\newtheorem{defn}[thm]{Definition}
\newtheorem{notation}[thm]{Notation}
\newtheorem{conv}[thm]{Convention} \newtheorem{rem}[thm]{Remark}

\newtheorem{propdfn}[thm]{Proposition-Definition}

\def\strutdepth{\dp\strutbox}
\def \ss{\strut\vadjust{\kern-\strutdepth \sss}}
\def \sss{\vtop to \strutdepth{
\baselineskip\strutdepth\vss\llap{$\diamondsuit\;\;$}\null}}

\def\strutdepth{\dp\strutbox}
\def \sst{\strut\vadjust{\kern-\strutdepth \ssss}}
\def \ssss{\vtop to \strutdepth{
\baselineskip\strutdepth\vss\llap{$\spadesuit\;\;$}\null}}

\def\strutdepth{\dp\strutbox}
\def \ssh{\strut\vadjust{\kern-\strutdepth \sssh}}
\def \sssh{\vtop to \strutdepth{
\baselineskip\strutdepth\vss\llap{$\heartsuit\;\;$}\null}}

\newcommand{\R}{\mathbb R}

\begin{document}

\begin{abstract}
  Let $F$ be a free group of rank $N\ge 2$, let $\mu$ be a geodesic
  current on $F$ and let $T$ be an $\mathbb R$-tree with a very
  small isometric action of $F$. We prove that the geometric intersection
  number $\langle T, \mu\rangle$ is equal to zero if and only if the
  support of $\mu$ is contained in the dual algebraic lamination
  $L^2(T)$ of $T$. Applying this result, we obtain a generalization of a theorem
  of Francaviglia regarding length spectrum compactness for currents
  with full support. We use the main result to obtain "unique ergodicity" type properties for the attracting and repelling fixed points of atoroidal iwip elements of $Out(F)$ when acting both on the compactified Outer Space and on the projectivized space of currents. We also show that the some of the translation length functions of any two "sufficiently transverse" very small $F$-trees is bilipschitz equivalent to the translation length function of an interior point of the Outer space. As another application, we define the notion of a
  \emph{filling} element in $F$ and prove that filling elements are
  "nearly generic" in $F$. We also apply our results to the notion
  of \emph{bounded translation equivalence} in free groups.
\end{abstract}

\thanks{The first author was supported by the NSF
  grants DMS-0404991 and DMS-0603921}


\maketitle

\tableofcontents

\section{Introduction}\label{intro}

The notion of a geometric intersection number between the free
homotopy classes of two essential closed curves on a compact surface
plays a fundamental role in the study of the Teichm\"uller space and
of the mapping class group. This notion naturally extends to the
notion of an intersection number between a closed curve and a
measured lamination (Thurston) as well as the notion of an
intersection form between two geodesic currents on a closed
hyperbolic surface (Bonahon). The space of measured laminations
naturally embeds into the space of currents. Thus Bonahon's
intersection form can be used to define the intersection number
between a measured lamination (or equivalently, its dual $\mathbb
R$-tree) and a geodesic current on a surface.

In the free group case, Culler-Vogtmann's \emph{Outer
space}~\cite{CV}, $cv(F)$, provides a natural analogue of the
Teichm\"uller space of a surface. The points of $cv(F)$ can be
thought of as minimal free discrete actions of $F$ on $\mathbb
R$-trees. One also often considers the projectivized outer space
$\mathbb P cv(F)$ that can be thought of as a subset
$CV(F)\subseteq cv(F)$ corresponding to actions where the
quotient graph has volume $1$.  Taking projective
classes of equivariant Gromov-Hausdorff limits of elements of $cv(F)$
leads to a natural Thurston-type compactification $\overline {CV}(F)=CV(F)\cup
\partial CV(F)$. The space $\overline {CV}(F)$ turns out to
consists precisely of all the projective classes of all minimal
\emph{very small} isometric actions of $F$ on $\mathbb
R$-trees~\cite{BF93,CL}. Recall that an isometric action of $F$ on
an $\mathbb R$-tree is \emph{very small} if nontrivial stabilizers of
non-degenerate arcs are maximal cyclic and if tripod stabilizers are
trivial.

Most automorpisms of $F$ (where $N\ge 3$) are not geometric, in
the sense that they are not induced by a self-homeomorphism of compact surface with boundary.
This leads to the breakdown, in the case of a free
group, of numerous symmetries and dualities from the hyperbolic
surface situation, and most equivalent notions from the world of
surfaces lead to distinct concepts in the free group setting.

The notion of a \emph{geodesic current} on $F$ (see
Definition~\ref{defn:current} below), and more generally, on a
word-hyperbolic group, is a measure-theoretic generalization of the
notion of a conjugacy class of a group element or of a free homotopy
class of a closed curve on a surface. Much of the motivation for
studying currents comes from the work of Bonahon about geodesic
currents on hyperbolic surfaces~\cite{Bo86,Bo88}.

The space $Curr(F)$ of all geodesic currents has a useful linear
structure and admits a canonical $Out(F)$-action. The space
$Curr(F)$ turns out to be a natural companion of the Outer space
and contains additional valuable information about the geometry and
dynamics of free group automorphisms. Examples of such applications
can be found in \cite{Bo91,CHL3,Fra,Ka1,Ka2,Ka3,KL,KN,KKS,Ma} and
other sources. Kapovich proved~\cite{Ka2} that for $F$ there does
not exist a natural symmetric analogue of Bonahon's intersection
number between two geodesic currents. However, there exists a natural
$Out(F)$-equivariant continuous \emph{intersection form}
\[
\langle\ ,\  \rangle : \overline{cv}(F)\times Curr(F)\to \mathbb
R,
\]
where $\overline{cv}(F)$ is the space of all very small minimal
isometric actions of $F$ on $\mathbb R$-trees and where
$Curr(F)$ is the space of geodesic currents on $F$. This
intersection form has several important features in common with
Bonahon's construction. In particular, if $T\in \overline{cv}(F)$
and $\eta_g\in Curr(F)$ is the \emph{counting current} for $g\in
F-\{1\}$ (see Definition~\ref{defn:rational} below), then
\[
\langle T, \eta_g\rangle=||g||_T,
\]
where $||g||_T$ is the \emph{translation length} of $g$ with respect to the
tree $T$, that is $||g||_T=\min_{x\in T} d(x,gx)$.
The intersection form was introduced in~\cite{Ka1,Ka2,L} for free
simplicial actions of $F$, that is for the non-projectivized Outer
Space $cv(F)$.  In a recent paper~\cite{KL4} we proved that the intersection form extends continuously to the
closure $\overline{cv}(F)$ of $cv(F)$ consisting of all minimal
very small isometric actions of $F$ on $\mathbb R$-trees.
Note that the projectivization $\mathbb P \overline{cv}(F)$ of
$\overline{cv}(F)$ is
exactly the compactification $\overline{CV}(F)$ of $CV(F)$.
For $T\in \overline{cv}(F)$ and for $\mu\in Curr(F)$  we will also call $\langle T,\mu\rangle$ and the \emph{geometric intersection number} of $T$ and $\mu$.

In general, if $T\in \overline{cv}(F)$
and if $\mu\in Curr(F)$ is approximated by rational currents as
$\mu=\lim_{i\to\infty} \lambda_i \eta_{g_i}$, where $g_i\in F$,
$\lambda_i\ge 0$, the geometric intersection number $\langle T,\mu\rangle$ can
be computed as
\[
\langle T,\mu\rangle=\lim_{i\to\infty} \lambda_i ||g_i||_T
\tag{$\ddag$}
\]
Ursula Hamenst\"adt~\cite{Ha} recently used our result from \cite{KL4} about the continuous extension of the intersection form to $\overline{cv}(F)$   as a key ingredient to prove that any non-elementary subgroup of $Out(F)$, where $N\ge 3$, has infinite dimensional second bounded cohomology group (infinite dimensional space of quasi-morphisms). This in turn has an application to proving that any homomorphism from any lattice in a higher-rank semi-simple Lie group to $Out(F)$, where $N\ge 3$, has finite image.

Very recently Bestvina and Feighn~\cite{BF08} used \cite{KL4}  to show
that for any finite collection $\phi_1,\dots,\phi_m\in Out(F_N)$ of {\em iwip} outer automorphisms of $F_N$ (``irreducible automorphisms with irreducible powers'', see Definition~\ref{defniwip})  there exists a $\delta$-hyperbolic complex $X = X(\phi_1,\dots,\phi_m)$ with an isometric $Out(F_N)$-action where each $\phi_i$ acts with a positive translation length.

Another crucial notion in the surface theory is that of
a geodesic lamination on a hyperbolic surface. In the free group
case, there is a companion notion of an abstract \emph{algebraic
lamination} which is understood as a closed $F$-invariant and
flip-invariant subset of
\[
\partial^2 F=\{(\xi_1,\xi_2): \xi_1,\xi_2\in \partial F, \xi_1\ne
\xi_2\}.
\]

A variation of this concept
was successfully exploited by Bestvina, Feighn and
Handel~\cite{BFH97} to analyze the dynamics of free group
automorphisms and the algebraic structure of subgroups of
$Out(F)$. That paper in turn played a key role in the eventual
proof of the Tits alternative for $Out(F)$ by Bestvina, Feighn and
Handel in \cite{BFH00,BFH05}.

Recently Coulbois, Hilion and Lustig~\cite{CHL1,CHL2,CHL3} gave a
detailed abstract treatment of the notion of an algebraic lamination
for free group. In particular, given a very small action of $F$ on
an $\mathbb R$-tree $T$, there is~\cite{CHL2} a naturally defined ``dual
lamination'' $L^2(T)$ on $F$ (see section 3 below).
In the special case where $T$ belongs to $cv(F)$, i.e. the $F$-action on $T$ is free and simplicial, then $L^2(T) = \emptyset$.
The relationship between $\mathbb R$-tree actions, laminations and
geodesic currents in the free group case turns out to be
considerably more delicate and complicated then for the case of
hyperbolic surfaces.  Investigating this relationship is an
important basic task in the study of $Out(F)$.

In the present paper we study the situation where the intersection
number between a tree and a current is equal to zero. Our main
result is:

\begin{thm}\label{thm:main}
Let $F$ be a finitely generated nonabelian free group with a very
small minimal isometric action on an $\mathbb R$-tree $T$. Let
$\mu\in Curr(F)$.

Then $\langle T,\mu\rangle=0$ if and only if
$supp(\mu)\subseteq L^2(T)$.
\end{thm}

Here $supp(\mu)\subseteq \partial^2 F$ is the \emph{support of
$\mu$} (see Section~\ref{sect:laminations} below for the
definition).

There is a known similar statement to Theorem~\ref{thm:main} in the surface context. 
Namely, suppose $\lambda$ and $\mu$ are measured geodesic laminations on a punctured hyperbolic surface $S$. The transverse measures on $\lambda$ and $\mu$ determines a geodesic currents $\hat \lambda$ and $\hat\mu$ on the surface $S$.
In this case Bonahon's intersection number between geodesic currents $\widehat \lambda$ coincides with Thurston's geometric intersection number: $i(\widehat \lambda, \widehat \mu)=i(\lambda,\mu)$. Moreover, since $\pi_1(S)$ is free, geodesic currents on $S$ in Bonahon's sense are also geodesic currents on $\pi_1(S)$ in the sense of the present paper.
Then $i(\lambda,\mu)=0$ if and only if the supports $supp(\lambda)$ and $supp(\lambda)$ of $\lambda$ and $\mu$ intersect in a common sublamination. If $T_\lambda$ denotes the ``dual" $\R$-tree transverse to $\lambda$ with metric defined by the transverse measure on $\lambda$ (see Ch. 11.12 in \cite{Kap} for details), then the definition in \cite{KL4} gives:
$$\langle T_\lambda, \widehat\mu \rangle = i(\lambda, \mu) = i(\widehat\lambda, \widehat\mu).$$
The lamination $L^2(T_\lambda)$ always contains $supp(\lambda)$, and if the latter {\em fills} the surface (i.e. all complementary components are contractible),  $L^2(T_\lambda)$ is precisely equal to the union of $supp(\lambda)$ with the (finite) set of diagonal leaves in the complementary surfaces of $supp(\lambda)$. In particular we see that in this case one has $\langle T_\lambda, \mu \rangle = 0$ if and only if $supp(\mu)$ is contained in $supp(\lambda) \subseteq supp(T_\lambda)$.

\smallskip

One of the main motivations and prospective uses for
Theorem~\ref{thm:main} is to analyze the \emph{intersection graph},
introduced by the authors in \cite{KL4} in order to study various
free group analogues of the curve complex. The intersection graph
$\mathcal I(F)$ is a bipartite graph with the vertex set $\mathbb
P\overline{cv}(F)\sqcup \mathbb PCurr(F)$ where $[T]\in  \mathbb
P\overline{cv}(F)$ and $[\mu]\in  \mathbb PCurr(F)$ are adjacent
in  $\mathcal I(F)$ if and only if
$\langle T,\mu \rangle=0$.
While it is generally not possible to define a good notion of an
intersection number between two conjugacy classes in a free group, one
can use the intersection form to generalize the notion of having
distance $\le 2$ in the standard curve complex. Thus one can define a
graph, whose vertices are conjugacy classes of primitive elements in
$F$ where two vertices $[a],[b]$ are adjacent if there exists $T\in
\overline{cv}(F)$ such that $\langle T,\eta_a\rangle=\langle T,\eta_b\rangle=0$, that is
$||a||_T=||b||_T=0$ (this graph is almost the same as the ``dual cut graph'' defined in \cite{KL4}). Taking a
dual point of view, one can think of an essential simple closed curve
on a surface as a splitting of the surface group over $\mathbb Z$,
where two curves are adjacent in the curve complex if and only if the
corresponding $\mathbb Z$-splittings have a common refinement. This
leads to the notion~\cite{KL4} of a \emph{cut graph} for $F$ whose vertices
are nontrivial splittings of $F$ as the fundamental group of a graph
of groups with a single edge and the trivial edge group, and where
adjacency again corresponds to having a common refinement. A variation
of this construction would declare two splittings to be adjacent if
there exists a nontrivial element in $F$ which is elliptic with
respect to both of them. All these (as others) natural analogues of
the curve complex can be studied by means of the intersection
graph. In \cite{KL4} we prove that for $N\ge 3$ the intersection graph
and all the free group analogues of the curve complex derived from it
have infinite diameter, by analyzing the action of iwip automorphisms
(see Definition \ref{defniwip}).  Iwip
automorphisms are also sometimes referred to as \emph{fully rreducible} in the literature. Note that Theorem~\ref{thm:main} gives a
characterization of adjacency in the intersection graph.

In this paper we apply Theorem~\ref{thm:main} to obtain a
generalization of a result of Francaviglia~\cite{Fra} about length
spectrum compactness for \emph{uniform currents} on $F$ (see
Definition~\ref{defn:na} below for the definition of the uniform current
with respect to a free basis of $F$). For $T\in cv(F)$ and $\mu\in
Curr(F)$, the \emph{automorphic length spectrum of $\eta$ with respect
  to $T$ } is the set
\[
{\mathcal S}_T(\mu):=\{ \langle T, \phi \mu\rangle: \phi\in
Out(F)\}=\{ \langle \phi T, \mu\rangle: \phi\in Out(F)\}\subseteq
\mathbb R.
\]
If $T=X(F,A)$ is the Cayley graph of $F$ corresponding to a free basis
$A$ of $F$ and $\mu=\nu_A$ is the uniform current corresponding to
$A$ then
\[
{\mathcal S}_T(\mu)=\{ \lambda_A(\phi): \phi\in Out(F)\}
\]
where $\lambda_A(\phi)$ is the \emph{generic stretching factor} of
$\phi$ with respect to $A$ (see \cite{KKS,Ka2} for definitions). Here we obtain:

\begin{thm}\label{thm:compact}
Let $\mu\in Curr(F)$ be a current with full support and let $T\in
cv(F)$. Then:

\begin{enumerate}
\item For any $C>0$ the set
\[
\{\phi\in Out(F): \langle T, \phi \mu\rangle\le C\}
\]
is finite.

\item The set ${\mathcal S}_T(\mu)$ is a discrete subset of $\mathbb R_{\ge
0}$.

\item Suppose $\phi_n\in Out(F)$ is an infinite sequence of distinct
elements such that for some $\lambda_n\ge 0$ and some $\mu'\in
Curr(F)$ we have $\lim_{n\to\infty} \lambda_n \phi_n \mu =\mu'$.
Then $\lim_{n\to\infty} \lambda_n=0$.
\end{enumerate}

\end{thm}

Francaviglia~\cite{Fra}, using very different methods, established
Theorem~\ref{thm:compact} for a certain class of currents with full
support, including uniform currents corresponding to free bases of $F$
and, more generally, Patterson-Sullivan currents corresponding to
points of $cv(F)$ (see~\cite{KN} for definitions).
Theorem~\ref{thm:compact} has applications to the ``ideal'' version of
the Whitehead algorithm for geodesic currents, as explained
in~\cite{Ka3}. Note that for $T\in cv(F)$ and $\mu\in Curr(F)$ that
does not have full support, the automorphic length spectrum ${\mathcal
  S}_T(\mu)$ need not be discrete. For example, if $\phi\in Out(F)$ is
an iwip which is atoroidal (that is, with no periodic conjugacy classes) then there exist
$\lambda>1$ and a nonzero ``eigencurrent'' $\eta\in Curr(F)$ such that
$\phi \eta=\lambda\eta$ (see \cite{Ma}). Then
$\phi^{-n}\eta=\frac{1}{\lambda^n}\eta$, so that for any $T\in cv(F)$
we have
\[
\langle T, \phi^{-n}\eta\rangle= \frac{1}{\lambda^n}\langle T,
\eta\rangle
\overset{n\to\infty}{\longrightarrow}
 0,
\]
and hence ${\mathcal S}_T(\eta)$ is not discrete.
A recent paper of R.~Sharp~\cite{Sharp} studies other dynamic-theoretic aspects related to generic stretching factors of free group automorphisms. A new paper of D.~Calegari and K.~Fujiwara~\cite{CF} investigates generalizations of generic stretching factors for different word metrics on word-hyperbolic groups.

We apply Theorem~\ref{thm:main} to establish ``unique ergodicity'' type properties for the attracting and repelling fixed points of atoroidal elements of $Out(F)$ in the compactified Outer Space and in the projectivized space of currents.

If $\phi\in Out(F)$ is an atoroidal iwip, then the (left) action of $\phi$ on $\overline{CV}(F)$ has exactly two distinct fixed points, an attracting fixed point $[T_+]$ and a repelling fixed point $[T_-]$ and similarly, the left action of $\phi$ on $\mathbb PCurr(F)$ has exactly two distinct fixed points, an attracting fixed point $[\mu_+]$ and a repelling fixed point $[\mu_-]$. In both cases every point distinct from the two fixed points lies on a ``North-South'' orbit with positive powers of $\phi$ making it converge to the attracting fixed point and with negative powers of $\phi$ making it converge to the repelling fixed point. These facts were established by Levitt and Lustig~\cite{LL} for the compactified Outer Space $\overline{CV}(F)$ and by  Reiner Martin~\cite{Ma} for $\mathbb PCurr(F)$.

In Section~\ref{sect:ue} we apply Theorem~\ref{thm:main} to establish unique-ergodicity type statements for $T_+$ and $\mu_+$. These results are summarized in the following:

\begin{thm}\label{thm:ns}
Let $\phi\in Out(F)$, where $N\ge 3$, be an atoroidal iwip. Let $[T_+]\in \overline{CV}(F)$ and $[\mu_+]\in \mathbb PCurr(F)$ be the attracting fixed points for the (left) actions of $\phi$ on $\overline{CV}(F)$ and $\mathbb PCurr(F)$ accordingly. Then:
\begin{enumerate}
\item If $[\mu]\in \mathbb PCurr(F)$ is such that $supp(\mu)\subseteq supp(\mu_+)$ then $[\mu]=[\mu_+]$.
\item If $[T]\in \overline{CV}(F)$ is such that $L^2(T_+)\subseteq L^2(T)$ then $[T]=[T_+]$.
\item Let $[\mu]\in \mathbb PCurr(F)$. Then $\langle T_+,\mu\rangle=0$ if and only if $[\mu]=[\mu_+]$.
\item Let $[T]\in \overline{CV}(F)$. Then $\langle T,\mu_+\rangle=0$ if and only if $[T]=[T_+]$.
\end{enumerate}
\end{thm}

We apply Theorem~\ref{thm:main} together with Theorem~\ref{thm:ns} to show that for two ``sufficiently transverse'' trees from $\overline{cv}(F)$ the sum of their translation length functions is bilipschitz equivalent to one coming from an interior point of $cv(F)$.
For two functions $f_1,f_2:F\to\mathbb R_{\ge 0}$ we write $f_1\sim f_2$ if there exists $C\ge 1$ such that for every $w\in F$ we have
\[
\frac{1}{C}f_1(w)\le f_2(w)\le Cf_1(w).
\]

We prove:
\begin{thm}\label{thm:bl}
Let $T_1,T_2\in \overline{cv}(F)$ be such that there does not exist $\mu\in Curr(F)$, $\mu\ne 0$ such that $\langle T_1,\mu\rangle=\langle T_2,\mu\rangle=0$. Then for every $T\in cv(F)$ we have $||.||_{T_1}+||.||_{T_2}\sim ||.||_T$.
\end{thm}
\begin{proof}
  Let $T\in cv(F)$ be arbitrary. Then for any nonzero $\mu\in Curr(F)$
  we have $\langle T,\mu\rangle>0$. This follows from the explicit
  definition of the intersection form in
  Proposition-Definition~\ref{pd:int-form} for the case where $T\in
  cv(F)$.

 Define a function $\mathcal J: Curr(F)-\{0\}\to\mathbb R$ as
\[
\mathcal J(\mu):=\frac{\langle T_1,\mu\rangle+\langle T_2,\mu\rangle}{\langle T,\mu\rangle}, \qquad \mu\in Curr(F), \mu\ne 0.
\]
Note that for every $\mu\in Curr(F),\mu\ne 0$ we have $0<\mathcal J(\mu)<\infty$ by assumption on $T_1,T_2$. Also, the function $\mathcal J$ is continuous by construction since the intersection form is continuous. Moreover, it is easy to see that for every $\mu\ne 0$ and every $c>0$ we have $\mathcal J(\mu)=\mathcal J(c\mu)$. Thus $\mathcal J$ factors through to a continuous strictly positive function $\mathcal J':\mathbb PCurr(F)\to\mathbb R$. Since $\mathbb PCurr(F)$ is compact, the function $\mathcal J'$ achieves a positive maximum and a positive minimum. Hence there exist $0<C_1<C_2<\infty$ such that $C_1\le \mathcal J(\mu)\le C_2$ for every $\mu\in Curr(F),\mu\ne 0$. Applying this fact to rational currents $\eta_g$, $g\in F-\{1\}$ we conclude that $||.||_{T_1}+||.||_{T_2}\sim ||.||_T$, as required.
\end{proof}

In the terminology of \cite{KL4}, the assumption of Theorem~\ref{thm:bl} says that the distance between $[T_1]$ and $[T_2]$ in the intersection graph $\mathcal I(F)$ is bigger than two. Using Theorem~\ref{thm:main} together with Theorem~\ref{thm:ns} we obtain the following corollary of Theorem~\ref{thm:bl}

\begin{cor}
Let $N\ge 2$. Then the following hold:
\begin{enumerate}
\item Let $T_1,T_2\in \overline{cv}(F)$ be such that $L^2(T_1)\cap L^2(T_2)=\emptyset$. Then for any $T\in cv(F)$ we have  $||.||_{T_1}+||.||_{T_2}\sim ||.||_T$.

\item Let $N\ge 3$ and let $\phi\in Out(F)$ be an atoroidal iwip. Let $[T_+],[T_-]\in \overline{CV}(F)$ be the attracting and repelling fixed points of $\phi$. Then for every $T\in cv(F)$ we have $||.||_{T_+}+||.||_{T_-}\sim ||.||_T$.

\item Let $N\ge 3$ and let $\phi,\psi\in Out(F)$ be atoroidal iwips such that their fixed points $[T_\pm(\phi)],[T_\pm(\psi)]\in \overline{CV}(F)$  are four distinct points. Then for every $T\in cv(F)$ we have $||.||_{T_+(\phi)}+||.||_{T_+(\psi)}\sim ||.||_T$.
\item Let $T_1,T_2\in \overline{cv}(F)$ be discrete simplicial trees with trivial arc stabilizers (thus, algebraically, they correspond to graph of groups decompositions of $F$ with trivial edge groups). Suppose that every nontrivial elliptic element for $T_1$ is hyperbolic for $T_2$ and that every nontrivial elliptic element for $T_2$ is hyperbolic for $T_1$ (that is the intersection of every conjugate of a vertex group of $T_1$ with every conjugate of a vertex group of $T_2$ is trivial). Then for any $T\in cv(F)$ we have  $||.||_{T_1}+||.||_{T_2}\sim ||.||_T$.
\end{enumerate}
\end{cor}

\begin{proof}
Part (1) follows directly from Theorem~\ref{thm:main} together with
Theorem~\ref{thm:bl}. Parts (2) and (3) easily follows from
Theorem~\ref{thm:bl} together with Theorem~\ref{thm:main} and Theorem~\ref{thm:ns}. We will give an argument for part (3) for concreteness and leave part (2) to the reader. The assumptions of part (3) together with Theorem~\ref{thm:ns} imply that the fixed points $[\mu_\pm(\phi)],[\mu_\pm(\psi)]$ of $\phi,\psi$ in $\mathbb PCurr(F)$ are four distinct points. Hence, again, by Theorem~\ref{thm:ns}, there does not exist $\mu\ne 0$ such that $\langle T_+(\phi),\mu\rangle=\langle T_+(\psi),\mu\rangle=0$. Hence by part (1) the conclusion of part (3) follows.

To see that part (4) holds it is not hard to show, using the explicit description of $L^2(T_1)$ and $L^2(T_2)$ obtained in Lemma~\ref{lem:L2} that under the assumptions of part (4) we have  $L^2(T_1)\cap L^2(T_2)=\emptyset$. Hence part (1) of the corollary applies.
\end{proof}

Theorem~\ref{thm:main} has some interesting applications to the
notions of a filling element (or a conjugacy class) and of bounded
translation equivalence. The concept of a filling curve on a surface
plays an important role in the surface theory. However, until now
there was no clear analogue of this notion in the free group
context, mainly because of the absence of a symmetric notion of an
intersection number between two conjugacy classes (or between two
currents) for a free group. We propose a natural asymmetric notion
of a filling conjugacy class here. It is easy to see that the free
homotopy class of an essential closed curve on a closed  hyperbolic
surface fills the surface if and only if this class has a positive
intersection number with every measured lamination on the surface,
or equivalently, if and only if the corresponding element of the
surface group has positive translation length for the dual $\mathbb
R$-tree of every measured lamination. By analogy, we will say that
an element $g\in F, g\ne 1$ \emph{fills} $F$ if for every very
small isometric action of $F$ on an $\mathbb R$-tree $T$ we have
$||g||_T>0$. More generally, we say that a current $\mu\in
Curr(F)$ \emph{fills} $F$ if for every very small isometric
action of $F$ on an $\mathbb R$-tree $T$ we have $\langle
T,\mu\rangle
>0$. Thus an element $g\in F$ fills $F$ if and only if the
corresponding ``counting current'' $\eta_g\in Curr(F)$ fills $F$.
As was explained to the authors by Vincent Guirardel, the
results of his paper~\cite{Gui} can be used to show that an element
$g\in F$ fills $F$ if and only if for every very small
simplicial action of $F$ on an $\mathbb R$-tree $T$ we have
$||g||_T>0$. However, we do not use this fact in this paper and work
directly with the definition of a filling element given above.

Unlike in the surface case, it is not at all obvious that in a free group
$F$ filling elements exist and even less clear why being filling is
a typical behavior. In fact, there is no known simple (or even complicated) explicit combinatorial criterion which guarantees that a given element of a free group is filling. Vincent Guirardel and Gilbert Levitt
showed us a special construction for producing filling elements in
free groups using iterated commutators and high powers. However, one
would still like to understand why being filling is an essentially
generic property for elements of free groups. Theorem~\ref{thm:main}
provides such an explanation. This theorem easily implies that every
current on $F$ with full support fills $F$. While counting
currents of elements of $F$ never have full support, it turns out
that any multiple of a counting current that is ``sufficiently close''
to a current with full support does fill $F$:

\begin{cor}\label{cor:fill1}
Let $F$ be a finitely generated nonabelian free group. Let $\mu\in
Curr(F)$ be a current with full support. Then:

\begin{enumerate}
\item The current $\mu$ fills $F$.

\item Suppose $g_i\in F$, $\lambda_i\ge 0$ are such that
  \[
\mu=\lim_{i\to\infty} \lambda_i \eta_{g_i}.
\]
Then there is some $i_0\ge 1$ such that for every $i\ge i_0$ the
element $g_i$ fills $F$.
\end{enumerate}

\end{cor}
Note that \emph{rational currents} (i.e. scalar multiples of
counting currents) are dense in $Curr(F)$, so that any current
$\mu$ with full support can be approximated by rational currents.
Corollary~\ref{cor:fill1} in turn implies that if $A$ is a free
basis of $F$ and $\xi$ is a random right-infinite freely reduced
word over $A^{\pm 1}$, then all sufficiently long initial segments
of $\xi$ give filling elements in $F$:

\begin{cor}\label{cor:fill2}
Let $F=F(A)$ be a finitely generated nonabelian free group with a
free basis $A$. Let $\mu_A$ be the uniform measure on $\partial F$
corresponding to $A$. Then there exists a set $R\subseteq \partial
F$ with the following properties:

\begin{enumerate}
\item We have $\mu_A(R)=1$.
\item For each $\xi\in R$ there is $n_0\ge 1$ such that for every $n\ge
  n_0$ the element $\xi_A(n)\in F$ fills $F$.
\end{enumerate}
\end{cor}
Here $\xi_A(n)\in F$ denotes the element of $F$ corresponding to
the initial segment of $\xi$ of length $n$. The \emph{uniform
measure} $\mu_A$ on $\partial F$ is a Borel probability measure (see
Definition~\ref{defn:ma} for a precise definition), such that a
$\mu_A$-random point of $\partial F$ corresponds to the intuitive
notion of a ``random'' right-infinite freely reduced word over $A$.

In \cite{KLSS} Kapovich, Levitt, Schupp and Shpilrain introduced the
notion of \emph{translation equivalence} in free groups, motivated
by the notions of hyperbolic equivalence and simple intersection
equivalence for curves in surfaces (see \cite{Le}). Recall that two elements
$g,h\in F$ are \emph{translation equivalent in $F$} if for every
very small action of $F$ on an $\mathbb R$-tree $T$ we have
$||g||_T=||h||_T$. The paper~\cite{KLSS} exhibited several sources
of translation equivalence and additional ones were found by Donghi
Lee~\cite{Lee}. The following is a natural generalization of the
notion of translation equivalence. We say that $g,h\in F$ are
\emph{boundedly translation equivalent in $F$}, denoted $g\equiv_b
h$, if there exists $C>0$ such that for for every very small action
of $F$ on an $\mathbb R$-tree $T$ we have
\[
\frac{1}{C} ||h||_T\le ||g||_T\le C ||h||_T.
\]
The following statement, together with Corollary~\ref{cor:fill1} and
Corollary~\ref{cor:fill2}, explains why, unlike translation
equivalence, bounded translation equivalence is an essentially
generic phenomenon in free groups:

\begin{cor}\label{cor:bte}
Let $g,h\in F$ be elements such that each of them fills $F$.
Then $g\equiv_b h$ in $F$.
\end{cor}

Note that while Corollary~\ref{cor:fill1} and
Corollary~\ref{cor:fill2} do explain why filling elements in $F$
are plentiful, they do not provide an explicit sufficient condition
for an element to be filling. Finding such a sufficient condition
that would be easily algorithmically verifiable and that would hold
for ``generic'' elements of $F$ remains an interesting problem. The
construction of Guirardel and Levitt mentioned above is both
explicit and algorithmic, but it relies on using iterated
commutators and large powers and thus is highly non-generic.

The paper is organized as follows. In Section~2 we review basic
definitions and notations related to outer space and geodesic
currents. In Section~3 we present background information regarding
algebraic laminations on free groups, introduce the notions of the
support of a current and the ``dual lamination'' associated with an
action of a free group on an $\mathbb R$-tree and establish some basic
facts about laminations and subgroups. In Section~4 we discuss the
``bounded back-tracking'' property for very small actions of $F$ on
$\mathbb R$-trees and its consequences. The bounded back-tracking
property for very small actions, established in \cite{GJLL}, is a key
tool in the present paper. In Section~5 we prove the ``if'' direction
of the main result, Theorem~\ref{thm:main}. Namely, in Theorem~\ref{thm:main2} we prove that if $supp(\mu)\subseteq
L^2(T)$ then $\langle T,\mu\rangle=0$. The ``only if'' direction of
Theorem~\ref{thm:main} turns out to be more difficult and requires
considering several different types of trees $T$ separately. In
Section~6 we establish the ``only if'' direction of
Theorem~\ref{thm:main} for the case of a tree $T\in \overline{cv}(F)$
with dense orbits. The case of a discrete $T\in \overline{cv}(F)$ is
dealt with in Section~7 and Section~8. Before dealing with the general
case, we develop some new machinery for restricting geodesic currents
to subgroups of $F$ in Section~9. In Section~10 we establish the
``only if'' direction of Theorem~\ref{thm:main} in the ``mixed'' or
general case of an arbitrary $T\in \overline{cv}(F)$. This is done in
Theorem~\ref{thm:main1} which completes the proof of Theorem~\ref{thm:main}
In Section~11 we apply the main result to prove
Theorem~\ref{thm:compact} about length spectrum compactness
(Theorem~\ref{thm:comp} in Section 11). In Section~12 we establish the
unique-ergodicity type results stated in Theorem~\ref{thm:ns}
above. In Section~13 we obtain applications of Theorem~\ref{thm:main}
to filling elements, filling currents and bounded translation
equivalence, stated in Corollary~\ref{cor:fill1},
Corollary~\ref{cor:fill2} and Corollary~\ref{cor:bte} above.

We are grateful to Vincent Guirardel, Gilbert Levitt, Pascal Hubert
and Chris Leininger for helpful comments and conversations. We are
also grateful to the referee for an extremely careful and detailed referee report
and for numerous useful suggestions.

\section{Geodesic currents}\label{sect:currents}

\begin{conv}\label{conv}
For the remainder of the article let $F$ be a finitely generated
nonabelian free group.
\end{conv}

Let $\partial F$ be the hyperbolic boundary of $F$ (see \cite{GH}
for background information about word-hyperbolic groups). We denote
\[
\partial^2F=\{(\xi_1,\xi_2): \xi_1,\xi_2\in \partial F, \text{ and }
\xi_1\ne \xi_2\}.
\]
Also denote by $\sigma_F: \partial^2 F\to \partial^2 F$ the ``flip''
map defined as $\sigma_F:(\xi_1,\xi_2)\mapsto (\xi_2,\xi_1)$ for
$(\xi_1,\xi_2)\in
\partial^2F$.

\begin{defn}[Simplicial charts]
A \emph{simplicial chart} on $F$ is an isomorphism $\alpha:
F\to\pi_1(\Gamma,x)$ where $\Gamma$ is a finite connected graph
without degree-one vertices and where is a vertex of $\Gamma$.
\end{defn}

From now on, when discussing graphs, for a graph $\Gamma$ we will
denote the set of vertices of $\Gamma$ by $V\Gamma$.

If $\alpha$ is a simplicial chart on $F$, it defines an
$F$-equivariant quasi-isometry between $F$ (with any word metric) and
$\widetilde \Gamma$, with the simplicial metric, that is where every
edge has length 1. Correspondingly, we
get canonical $F$-equivariant homeomorphisms $\widetilde \alpha:
\partial F\to \partial \widetilde \Gamma$ and $\widehat \alpha:
\partial^2 F\to \partial^2 \widetilde \Gamma$, that do not depend on
the choice of a word metric for $F$.  If $\alpha$ is fixed, we will
usually use these homeomorphisms to identify $\partial F$ with
$\partial \widetilde \Gamma$ and $\partial^2 F$ with $\partial^2
\widetilde \Gamma$ without additional comment.

\begin{defn}[Cylinders]
Let $\alpha: F\to\pi_1(\Gamma,x)$ be a simplicial chart on $F$. For
a nontrivial reduced edge-path $\gamma$ in $\widetilde \Gamma$
denote by $Cyl_{\widetilde \Gamma}(\gamma)$ the set of all
$(\xi_1,\xi_2)\in
\partial^2 F$ such that the bi-infinite geodesic from $\tilde
\alpha(\xi_1)$ to $\tilde \alpha(\xi_2)$ contains $\gamma$ as a
subpath.

We call $Cyl_{\widetilde \Gamma}(\gamma)\subseteq \partial^2 F$ the
\emph{two-sided cylinder corresponding to $\gamma$}.

It is easy to see that $Cyl_{\widetilde \Gamma}(\gamma)\subseteq
\partial^2 F$ is both compact and open. Moreover, the collection of
all such cylinders, where $\gamma$ varies over all nontrivial reduced
edge-paths in $\widetilde \Gamma$, forms a basis of open sets in
$\partial^2 F$.
\end{defn}

\begin{defn}[Geodesic currents]\label{defn:current}
A \emph{geodesic current} on $F$ is a positive Radon measure (that is
a Borel measure which is finite on compact sets) on
$\partial^2 F$ that is $F$-invariant and $\sigma_F$-invariant. The
set of all geodesic currents on $F$ is denoted by $Curr(F)$. The set
$Curr(F)$ is endowed with the weak topology which makes it into a
locally compact space.

Specifically, let $\alpha: F\to\pi_1(\Gamma,x)$ be a simplicial
chart on $F$. Let $\mu_n,\mu\in Curr(F)$. It is not hard to
show~\cite{Ka2} that $\lim_{n\to\infty}\mu_n=\mu$ in $Curr(F)$ if
and only if for every nontrivial reduced edge-path $\gamma$ in
$\widetilde\Gamma$ we have
\[
\lim_{n\to\infty} \mu_n(Cyl_{\widetilde
\Gamma}(\gamma))=\mu(Cyl_{\widetilde \Gamma}(\gamma)).
\]

Let $\mu\in Curr(F)$ and let $v$ be a nontrivial reduced edge-path
in $\Gamma$. Denote
\[
\langle v, \mu\rangle_{\alpha}:=\mu(Cyl_{\widetilde
\Gamma}(\gamma)),
\]
where $\gamma$ is any edge-path in $\widetilde \Gamma$ that is
labelled by $v$ (i.e. which is a lift of $v$ to $\widetilde\Gamma$).  Since $\mu$ is $F$-invariant, this definition does
not depend on the choice of a lift $\gamma$ of $v$.
\end{defn}

\begin{notation}
Let $A=\{a_1,\dots, a_k\}$ be a free basis of $F$ and let $\alpha$ be the simplicial
chart on $F$ corresponding to $A$. That is, $\alpha:F\to
\pi_1(\Gamma,x)$, where $\Gamma$ is a wedge of $k$ loop-edges at a
single vertex $x$, where the edges are labelled by $a_1,\dots
a_k$. The map $\alpha$ sends a freely reduced word $v\in F(A)$ to the
edge-path in $\Gamma$ labelled by $v$. Then $\widetilde \Gamma=X(F,A)$
is the Cayley graph of $F$ with respect to $A$.

In this case, for $v\in F(A)$ and $\mu\in Curr(F)$, we will denote $\langle
\alpha(v),\mu\rangle_\alpha$ by $\langle v, \mu \rangle_A$.
\end{notation}

\begin{notation} For any $g\in F, g\ne 1$ denote $g^{\infty}=\lim_{n\to\infty}
g^n$ and $g^{-\infty}=\lim_{n\to\infty} g^n$ so that
$(g^{-\infty},g^{\infty})\in \partial^2F$.

Also, for any $g\in F$ we will denote by $[g]_F$ or just by $[g]$
the conjugacy class of $g$ in $F$.
\end{notation}

\begin{defn}[Counting and Rational Currents]\label{defn:rational}
Let $g\in F$ be a nontrivial element that is not a proper power in
$F$. Put
\[
\eta^F_g=\sum_{h\in [g]_F}
\left(\delta_{(h^{-\infty},h^{\infty})}+\delta_{(h^{\infty},h^{-\infty})}\right).
\]

Let $\mathcal R(g)$ be the collection of all $F$-translates of
$(g^{-\infty},g^{\infty})$ and $(g^{\infty},g^{-\infty})$ in
$\partial^2 F$. Note that if $u\in F$ and $h=ugu^{-1}$ then
$u(g^{-\infty}, g^\infty)=(h^{-\infty}, h^\infty)$.  Therefore
\[
\eta^F_g=\sum_{(x,y)\in \mathcal R(g)} \delta_{(x,y)},
\]
and hence $\eta^F_g$ is $F$-invariant and flip-invariant, that is
$\eta_g\in Curr(F)$.

Let $g\in F$ be an arbitrary nontrivial element. Write $g=f^m$ where
$m\ge 1$ and $f\in F$ is not a proper power. Put
$\eta^F_g:=m\eta^F_f$.

We call $\eta^F_g\in Curr(F)$ the \emph{counting current
corresponding to $g$}. Positive scalar multiples of counting
currents are called \emph{rational currents}. If the ambient group
$F$ is fixed, we will often denote $\eta^F_g$ by $\eta_g$.
\end{defn}

It is easy to see that if $[g]_F=[h]_F$ then $\eta^F_g=\eta^F_h$ and
$\eta^F_g=\eta^F_{g^{-1}}$.

The following statement is an important basic fact regarding rational currents:
\begin{prop}\cite{Ka1,Ka2}
The set of all rational currents is dense in the space  $Curr(F)$.
\end{prop}

\begin{defn}[Cyclic paths and cyclic words]
  A \emph{cyclic path} or \emph{circuit} in $\Gamma$ is an immersion
  graph-map $c:\mathbb S\to \Gamma$ from a simplicially subdivided
  oriented circle $\mathbb S$ to $\Gamma$. Let $u$ be an edge-path in
  $\Gamma$. An \emph{occurrence of $u$ in $c$} is a vertex of $\mathbb
  S$ such that, going from this vertex in the positive direction along
  $\mathbb S$, there exists an edge-path in $\mathbb S$ (not
  necessarily simple and not necessarily closed) which is labelled by
  $u$, that is, which is mapped to $u$ by $c$. We denote by $\langle
  u,c\rangle$ the number of occurrences of $u$ in $c$.

  If $A$ is a free basis of $F$ and $\Gamma$ is a bouquet of edges
  labelled by the elements of $A$, then a cyclic path in $\Gamma$ can
  also be thought of as a \emph{cyclic word} over $A$. A \emph{cyclic
    word} is a cyclically reduced word in $F(A)$ written on a simplicially subdivided
    circle (where every positively oriented edge is labelled by an
    element of $A$) in the clockwise-direction without a specified
    base-point. The number of occurrences of $v\in F(A)$ in a cyclic
    word $w$ over $A$ is denoted $\langle v,w\rangle_A$.
  \end{defn}

Let $\alpha:F\to \pi_1(\Gamma,x)$ be a simplicial chart for $F$.
Then every nontrivial conjugacy class $[g]_F$ is represented by a
unique reduced cyclic edge-path $w_g$ in $\Gamma$.

If $w$ is a cyclic edge-path in $\Gamma$ and $v$ is an edge-path in
$\Gamma$, denote by $\langle v,w\rangle_\alpha$ the number of
occurrences of $v$ in $w$. We will also occasionally use the
notation $\langle v^{\pm 1},w\rangle_\alpha:=\langle
v,w\rangle_\alpha+\langle v^{-1},w\rangle_\alpha$.

It is not hard to see that the definition of a counting current can
be reinterpreted as follows~\cite{Ka2}:

\begin{lem}
Let $\alpha:F\to \pi_1(\Gamma,x)$ be a simplicial chart for $F$. Let
$g\in F$ be a nontrivial element and let $w_g$ be the reduced cyclic
path in $\Gamma$ representing $[g]_F$.

Then for every reduced edge-path $v$ in $\Gamma$ we have
\[
\langle v, \eta_g\rangle_\alpha=\langle v^{-1},
\eta_g\rangle_\alpha=\langle v, w_g\rangle_\alpha+\langle
v^{-1},w_g\rangle_\alpha=\langle v^{\pm 1},w\rangle_\alpha.
\]
\end{lem}

\begin{propdfn}[Intersection Form]\label{pd:int-form}\cite{KL4}
Let $F$ be a finitely generated nonabelian free group and with a
very small minimal isometric action on an $\mathbb R$-tree $T$. Let
$\mu\in Curr(F)$. Let $\mu=\lim_{i\to\infty}\lambda_i\eta_{g_i}$ where $\lambda_i\ge
0$ and $g_i\in F$. Then:
\begin{enumerate}
\item The limit
  \[
\lim_{i\to\infty} \lambda_i||g_i||_T
\]
exists and does not depend on the choice of a sequence of
rational currents $\lambda_i\eta_{g_i}$ approximating $\mu$. We call
this limit the \emph{length of $\mu$ with respect to $T$} or the
\emph{geometric intersection number of $T$ and $\mu$} or the
\emph{length of $\mu$ with respect to $T$} and denote it by $\langle
T,\mu\rangle$ or by
$||\mu||_T$.
\item Let $\alpha:F\to\pi_1(\Gamma,x)$ be a simplicial chart on $F$
  and let $\mathcal L$ be a metric structure on $\Gamma$. That is,
  each oriented edge $e\in E\Gamma$ is assigned a length $\mathcal L(e)>0$ so
  that $\mathcal L(e)=\mathcal L(e^{-1})$ for every $e\in E\Gamma$.
  Let $T$ be the $\mathbb R$-tree obtained by giving each edge in
  $\widetilde \Gamma$ the same length as that of its projection in
  $\Gamma$. Thus $F$ acts on $T$ freely and discretely by isometries,
  via $\alpha$.  Then

\[
\langle T, \mu\rangle=\frac{1}{2}\sum_{e\in E\Gamma} \mathcal L(e)
\langle e, \mu\rangle_{\alpha}.
\]
\item The function
  \[
\langle\ , \ \rangle: \overline{cv}(F)\times Curr(F)\to \mathbb R
\]
is continuous, $Out(F)$-invariant, $\mathbb R_{\ge 0}$-linear with
respect to the second argument and $\mathbb R_{\ge 0}$-homogeneous
with respect to the first argument.

\end{enumerate}
\end{propdfn}
Proposition-Definition~\ref{pd:int-form} for free simplicial actions was
obtained in ~\cite{Ka2,L}. Recently, Kapovich and Lustig~\cite{KL4} generalized this result
to the case of arbitrary very small actions and proved Proposition-Definition~\ref{pd:int-form} in the
form stated above.

Let $A$ be a free basis of $F$ and let $X(F,A)$ be the Cayley graph of $F$
corresponding to $A$. Thus $X(F,A)\in cv(F)$.
 For a current $\mu\in Curr(F)$ denote
$||\mu||_A:=||\mu||_{X(F,A)}=\langle X(F,A), \mu\rangle$.

\begin{defn}[Uniform current]\label{defn:na}
Let $A$ be a free basis of $F$ and let $k\ge 2$ be the rank of $F$.
Consider the simplicial chart $\alpha$ on $F$ corresponding to $A$.
The \emph{uniform current corresponding to $A$}, denoted by $\nu_A$
is the current defined by
\[
\langle v, \nu_A\rangle_A=\frac{1}{2k(2k-1)^{n-1}},
\]
where $v\in F(A)$ is a nontrivial freely reduced word and where $|v|_A=n$.
\end{defn}

It is not hard to show that $\nu_A$ is indeed a geodesic current on
$F$ and we refer the reader to \cite{Ka2,KKS} for a more detailed
discussion.

\section{Laminations}\label{sect:laminations}

Recall that, as specified in Convention~\ref{conv},  $F$  is a
finitely generated nonabelian free group. We refer the reader to
\cite{CHL1,CHL2,CHL3} for detailed background material on algebraic
laminations in the context of free groups. We shall only state some
basic definitions and facts here.

\begin{defn}[Algebraic laminations]
An \emph{algebraic lamination} on $F$ is a closed $F$-invariant and
flip-invariant subset $L\subseteq \partial^2 F$.

Denote by ${\Lambda^2}(F)$ the set of all algebraic laminations on
$F$.
\end{defn}

\begin{defn}[Laminary language of an algebraic lamination]

Let $\alpha: F\to \pi_1(\Gamma,x)$ be a simplicial chart of $F$ and
let $X=\widetilde \Gamma$. Recall that $\alpha$ defines a canonical
$F$-equivariant homeomorphism $\widetilde{\alpha}: \partial F\to \partial X$

Let $B$ be the set of oriented edges of $\Gamma$.
Let $L\in {\Lambda^2}(F)$. The \emph{laminary language of $L$,
corresponding to $\alpha$}, denoted $L_{\alpha}$,  is defined as the
set of all reduced edge-paths $v$ in $\Gamma$ such that there exists
a bi-infinite geodesic $\gamma$ in $X$  with endpoints
$\widetilde{\alpha}(\xi_1),\widetilde\alpha(\xi_2)\in \partial X$ for some
$(\xi_1,\xi_2)\in L$, such that $\gamma$ contains a subsegment
labelled by $v$. We can think of $L_{\alpha}$ as a subset of
$B^\ast$, where $B^\ast$ is the set of all words in the alphabet $B$.

If $\alpha$ is a simplicial chart on $F$ corresponding to a free
basis $A$ of $F$ (where $\Gamma$ is the wedge of circles labelled by
elements of $A$ and where $X$ is the Cayley graph of $F$ with
respect to $A$), we will denote $L_{\alpha}$ by $L_A$. In this case
$L_A$ is a set of freely reduced words in $F=F(A)$.
\end{defn}

Note that in \cite{CHL1,CHL2,CHL3} laminary languages are defined
only with respect to a free basis of a free group. However, it is
easy to see that the definition and the basic results listed here
extend to an arbitrary simplicial chart, that is not necessarily a
wedge of loop-edges.

\begin{prop}\cite{CHL1}
Let $\alpha: F\to \pi_1(\Gamma,x)$ be a simplicial chart of $F$.
Let $L,L'\in {\Lambda^2}(F)$. Then $L\subseteq L'$ if and only
if $L_{\alpha}\subseteq L'_{\alpha}$.
\end{prop}

Let $H\le F$ be a finitely generated subgroup. Then the inclusion of
$H$ in $F$ extends to a canonical $H$-equivariant topological
embedding $i_H: \partial H\to
\partial F$. We will usually suppress this embedding and write that
$\partial H\subseteq \partial F$ and $\partial^2 H\subseteq
\partial^2 F$.

\begin{propdfn}
Let $H\le F$ be a finitely generated subgroup and let $L\in
{\Lambda^2}(H)$ be an algebraic lamination on $H$. Then
\[
i_\Lambda(L):=\overline{\underset{f\in F}{\bigcup} f L}
\]
is a closed $F$-invariant subset of $\partial^2 F$, that is, an
algebraic lamination on $F$. Thus $i_\Lambda(L)\in {\Lambda^2}(F)$.
\end{propdfn}

\begin{defn}[Lamination defined by a tree action]\cite{CHL1,CHL2,CHL3}\label{defn:dual}
Let $F$ be a finitely generated free group acting isometrically on
an $\mathbb R$-tree $T$. The \emph{lamination on $F$ corresponding
to the action of $F$ on $T$}, denoted $L^2(T)\in {\Lambda^2}(F)$, is
defined as follows.

Choose a simplicial chart $\alpha: F\to \pi_1(\Gamma,x)$ and fix the
corresponding identification of $\partial^2 F$ and $\partial^2
\widetilde \Gamma$.

Then for $(\xi_1,\xi_2)\in \partial^2 F$ we have $(\xi_1,\xi_2)\in
L^2(T)$ if and only if for every $\epsilon>0$ and every reduced
edge-path $v$ in $\Gamma$ labelling some subsegment of the
bi-infinite geodesic joining $\xi_1$ and $x_2$ in $\widetilde
\Gamma$ there is some reduced and cyclically reduced closed path $w$
in $\Gamma$ containing $v$ as a subpath such that $||w||_T\le
\epsilon$.

It is not hard to show (see \cite{LL,CHL2}) that this definition of
$L^2(T)$ does not depend on the choice of a simplicial chart on $F$.
\end{defn}

\begin{defn}[Support of a current]\label{defn:support}
Let $\mu\in Curr(F)$. Then we define the \emph{support of $\mu$},
denoted $supp(\mu)$, as $\partial^2F-\mathcal U$ where $\mathcal U$
is the union of all open subsets $U\subseteq \partial^2 F$ such that
$\mu(U)=0$. It is easy to see that $supp(\mu)\subseteq \partial^2 F$
is both closed and $F$-invariant, so that $supp(\mu)\in
\Lambda^2(F)$. We say that $\mu\in Curr(F)$ has \emph{full support}
if $supp(\mu)=\partial^2F$.

Let $\alpha:F\to\pi_1(\Gamma,x)$ be a simplicial chart on $F$, let
$L=supp(\mu)\in \Lambda^2(F)$ and let $B=E\Gamma$. We will denote
the laminary language $L_\alpha\subseteq B^\ast$ by
$supp_{\alpha}(\mu)$. Thus $supp_{\alpha}(\mu)$ consists of all
reduced paths $v$ in $\Gamma$ such that $\langle
v,\mu\rangle_\alpha>0$.

In the case where the simplicial chart $\alpha$ is defined by a free
basis $A$ of $F$, we will denote $L_\alpha$ by $supp_A(\mu)$.
\end{defn}

For example, it is obvious that for any free basis $A$ of $F$ the
uniform current $\nu_A$ has full support. The following lemma is an
easy corollary of the definitions:
\begin{lem}
Let $\alpha:F\to\pi_1(\Gamma,x)$ be a simplicial chart on $F$ and
let $\mu\in Curr(F)$. Let $L=supp(\mu)\in {\Lambda^2}(F)$. Then for
a reduced edge-path $v$ in $\Gamma$ we have $v\in L_{\alpha}$ if and
only if $\langle v, \mu\rangle_{\alpha}>0$.
\end{lem}

\begin{notation}
Let $F$ be a finitely generated free group with a free basis $A$. We
will denote by $X(F,A)$ the Cayley graph (which happens to be a
tree) of $F$ with respect to $A$.

For a nontrivial finitely generated subgroup $H\le F$ denote by
$X_H$ the smallest $H$-invariant subtree of $X(F,A)$.

We also denote by $\Gamma_H$ the Stallings subgroup graph of $H$
with respect to $A$ (see~\cite{KM} for a detailed discussion about
Stallings subgroup graphs). Recall that $\Gamma_H$ can be obtained
as follows. Let $x\in X_H$ be the closest to $1\in F$ vertex of
$X_H$. Then $X'_H=X_H\bigcup \underset{f\in F}{\cup} f[1,x]$ is an
$H$-invariant subtree of $X(F,A)$. Then $\Gamma_H=X'_H/H$. Recall
that every oriented edge in $\Gamma_H$ has a label $a\in A^{\pm 1}$
which comes from the corresponding label of the edge in $X'_H$. The
image, under the quotient map $X'_H\to X'_H/H=\Gamma_H$ of the
vertex $1\in F$ in $\Gamma_H$ is the base-vertex of $\Gamma_H$.

For a finite connected graph $\Gamma$ with a nontrivial fundamental
group, the \emph{core of $\Gamma$}, denoted $Core(\Gamma)$, is the
unique smallest subgraph $\Delta$ of $\Gamma$ such that $\Delta$ is
homotopically equivalent to $\Gamma$. Note that $\Delta$ has no
degree-one vertices and that $\Gamma$ is equal to the union of
$\Delta$ and a finite (possibly empty) collection of trees attached
to some vertices of $\Delta$.
\end{notation}

\begin{conv}
For the remainder of this section, let $H\le F=F(A)$ be a finitely
generated subgroup. Let $\Gamma_H$ be the Stallings subgroup graph
of $H$ with respect to $A$ and let $\Delta_H=Core(\Gamma_H)$.

Let $x\in \Delta_H$ be the closest vertex of $\Delta_H$ to the basepoint $y$ of
$\Gamma_H$. Let $u\in F(A)$ be the label of the
segment joining $y$ to $x$ in $\Gamma_H$. Then we have a canonical
isomorphism $\alpha_{H,A}:H\to \pi_1(\Gamma_H,y)$, where
$\alpha_{H,A}^{-1}$ sends a loop $\gamma$ at $x$ in $\Delta_H$ to
$uwu^{-1}$, where $w\in F(A)$ is the label of $\gamma$. Thus the
pair $(\Gamma_H,y)$ defines a canonical isomorphism
$\alpha_{H,A}:H\to \pi_1(\Gamma_H,y)$ which is a simplicial chart
for $H$. Moreover, a copy of $\widetilde{\Delta}_H$ is contained in
$X(F,A)$ and coincides with the smallest $H$-invariant subtree of
$X(F,A)$.

\end{conv}

The following lemma is easily established by compactness argument
since the graph $\Delta_H$ is finite.

\begin{lem}\label{lem:H}
Let $L\in {\Lambda^2}(H)$ and let $L'=i_\Lambda(L)\in \Lambda^2(F)$.
Then for $v\in F(A)$ we have $v\in L'_A$ if and only if there is a
bi-infinite reduced path $\gamma$ in $\Delta_H$ whose $F(A)$-label
contains $v$ as a subword, such that some bi-infinite geodesic in
$\widetilde \Delta_H$, whose pair of endpoints is an element of $L$,
projects to $\gamma$.
\end{lem}

Lemma~\ref{lem:H} in turn easily implies:
\begin{lem}
The following hold:
\begin{enumerate}
\item Let $(\xi_1,\xi_2)\in \partial^2 F$. Then $(\xi_1,\xi_2)\in
  i_{\Lambda}(\partial^2 H)$ if and only if the label of the
  bi-infinite geodesic from $\xi_1$ to $\xi_2$ in $X(F,A)$ is the
  $F(A)$-label of some bi-infinite reduced edge-path in
  $Core(\Gamma_H)$.
\item The subset $\underset{f\in F}{\bigcup} f\partial^2H\subseteq
  \partial^2 F$ is closed, and hence
  \[
i_{\Lambda}(\partial^2 H)=\underset{f\in F}{\bigcup} f\partial^2H.
\]
\end{enumerate}

\end{lem}

\section{Bounded back-tracking}

\begin{defn}[Bounded Back-tracking Constant]
Let $F=F(A)$ be a finitely generated free group acting isometrically
on an $\mathbb R$-tree $T$. Let $X(F,A)$ be the Cayley graph of $F$
with respect to $A$. Let $p\in T$. There is a unique $F$-equivariant
map $i_p: X(F,A)\to T$ that is linear on edges on $X(F,A)$ and with
$i_p(1)=p$. The \emph{Bounded Back-tracking Constant corresponding
to $A$, $T$ and $p$}, denoted $BBT_{T,p}(A)$, is the infimum of all
$C>0$ such that for any $Q,R\in X(F,A)$ the image $i_p([Q,R])$ of
$[Q,R]\subseteq X(F,A)$ is contained in the $C$-neighborhood of
$[i_p(Q),i_p(R)]$. The \emph{Bounded Back-tracking Constant
corresponding to $A$ and $T$}, denoted $BBT_{T}(A)$, is the infimum
over all $p\in T$ of $BBT_{T,p}(A)$.
\end{defn}

An useful result of \cite{GJLL} states:

\begin{prop}\label{prop:BBT}
Let $F$ be a finitely generated nonabelian free group with a very
small isometric minimal action on an $\mathbb R$-tree $T$. Let $A$
be a free basis of $F$ and let $p\in T$. Then $BBT_{T,p}(A)<\infty$.
\end{prop}

The following is an easy corollary of the definitions:
\begin{lem}\label{lem:LL}
Let $F$ be a finitely generated nonabelian free group with a very
small isometric minimal action on an $\mathbb R$-tree $T$. Let $A$
be a free basis of $F$ and let $p\in T$.

Suppose $BBT_{T,p}(A)<C$. Let $u=u_1\dots u_m$ be a freely reduced
product of freely reduced words in $F(A)$, where $m\ge 1$. Then the
following hold:
\begin{enumerate}
\item Let $w\in F(A)$ be cyclically reduced. Then
  \[
\big| ||w||_T-d_T(p,wp) \big|\le 2C.
\]
\item Let $u=u_1\dots u_m$ be a freely reduced product of freely
  reduced words in $F=F(A)$, where $m\ge 1$.  Then we have
  \[\left| d_T(p,up)-\sum_{i=1}^m d_T(p,u_ip) \right|\le 2mC.\]
\item Suppose, in addition, that $u,u_1,\dots, u_m$ are cyclically
  reduced in $F(B)$. Then
  \[
\left| ||u||_T-\sum_{i=1}^m ||u_i||_T  \right|\le 4mC.
\]
\end{enumerate}
\end{lem}

Levitt and Lustig~\cite{LL} prove the following:
\begin{prop}\label{prop:LL}
Let $F$ be a finitely generated nonabelian free group with a very
small isometric action on an $\mathbb R$-tree $T$ with dense orbits.

Then for any $\epsilon>0$ and any $p\in T$ there exists a free basis
$B$ of $F$ such that the following hold:
\begin{enumerate}
\item We have $BBT_{T,p}(B)<\epsilon$.

\item For every $b\in B$ we have $d_T(p,bp)<\epsilon$.
\end{enumerate}
\end{prop}

If $A$ is a free basis of a free group $F(A)$, for $u\in F(A)$ we
denote by $|u|_A$ the freely reduced length of $u$ with respect to $A$
and we denote by $||u||_A$ the cyclically reduced length of $u$ with
respect to $A$. Thus $||u||_A$ is the translation length of $u$ on the
Cayley tree $X(F(A),A)$ of $F(A)$.

\begin{propdfn}[Double Bounded Cancellation Constant]
Let $A,B$ be two free bases of a finitely generated free group $F$.
Then there exist integers $1\le l \le n$ with the following
properties:

\begin{enumerate}
\item Let $u,v$ be freely reduced words in $F(A)$ such that the word
  $uv$ is freely reduced. Let $u',v'$ be freely reduced words in
  $F(B)$ representing $u,v$ accordingly. Then the maximal terminal
  segment of $u'$ that freely cancels in the product $u'\cdot v'$ has
  length $\le l$.
\item If $u$ is a cyclically reduced word in $F(A)$ and $u'$ is the
  freely reduced form of $u$ in $F(B)$ then
  $\big|||u'||_B-|u|_B\big|\le 2l$.
\item Let $w=uw_0v$ be a freely reduced product of freely reduced words
  in $F(A)$ and suppose that $|u|_A,|v|_A>n$. Let $x$, $y$ be freely
  reduced words in $F(A)$ such that the product $xwy$ is reduced as
  written. Let $x',w',y'$ be the freely reduced forms of $x,w,y$ in
  $F(B)$. Let $w''$ be the maximal subword of $w'$ that is not
  affected by the free cancellations in the product $x'w'y'$. Then
  $w''$ is nonempty.

Let $z_1',z_2'$ be freely reduced words in $F(B)$ such that
$z_1'w''z_2''$ is freely reduced as written. Let $z_1,z_2,\tilde w$
be the freely reduced words in $F(A)$ representing $z_1',z_2',w''$
accordingly. Then the maximal subword $W$ of $\tilde w$ that is not
affected by the free cancellations in $z_1\tilde w z_2$ is nonempty
and, in addition, $W$ contains $w_0$ as a subword. Moreover, if
$z_1'w''z_2'$ is cyclically reduced in $F(B)$, then the maximal
subword of $\tilde w$ that is not affected by the free and cyclic
reduction of $z_1\tilde w z_2$, is nonempty and contains $w_0$ as a
subword.
\end{enumerate}

The existence of $l\ge 1$ as above follows from the bounded
cancellation lemma~\cite{Coo}. It is not hard to deduce the
existence of $n\ge 1$ from the existence of $l$.

We refer to the smallest integer $n=n_{A,B}\ge 1$ satisfying
condition (3) above as the \emph{double bounded cancellation
constant for $A$, $B$}.

 Note that the definition of $n_{A,B}$
given above is not symmetric and it is possible that $n_{A,B}\ne
n_{B,A}$. Note also that conditions (1), (2) in the above definition
hold with $l=n_{A,B}$.

\end{propdfn}

\begin{defn}
Let $A,B$ be two free bases of a finitely generated free group $F$.

Let $C=\max_{a\in A}|a|_B$.  We call $C=C_{A,B}$ the \emph{Lipshitz
constant for $A$, $B$}.
\end{defn}

\section{The ``if'' case of the main result}\label{sect:if}

In this section we will establish the ``if'' implication in
Theorem~\ref{thm:main}. Namely, we will prove that for an arbitrary
very small action of $F$ on an $\mathbb R$-tree $T$ the condition
$supp(\mu)\subseteq L^2(T)$ implies that $\langle T,\mu\rangle=0$.

\begin{rem}\label{rem:norm}
Suppose $\mu\in Curr(F), \mu\ne 0$ and let
\[
\mu=\lim_{i\to\infty} \lambda_i \eta_{g_i}
\]
where $\lambda_i\ge 0$ and $g_i\in F$. Then for $\lambda_i'\ge 0$ we
have
\[
\mu=\lim_{i\to\infty} \lambda_i' \eta_{g_i} \quad\iff\quad
\lim_{i\to\infty} \frac{\lambda_i'}{\lambda_i}=1.
\]
In particular (see \cite{Ka1,Ka2}), if $A$ is a free basis of $F$, and $\mu\in Curr(F)$ is such that
$||\mu||_A=1$ then for any $\lambda_i\ge 0$, $g_i\in F$ with $\displaystyle \mu=\lim_{i\to\infty} \lambda_i \eta_{g_i}$ we have:
\[
\mu=\lim_{i\to\infty} \frac{1}{||g_i||_A} \eta_{g_i}.
\]
\end{rem}

\begin{thm}\label{thm:main2}
Let $F$ be a finitely generated nonabelian free group. Let $T$ be a
$\mathbb R$-tree with a very small minimal isometric action of $F$.
Let $\mu\in Curr(F)$ be such that $supp(\mu)\subseteq L^2 (T)$. Then $\langle T,\mu\rangle=0$.
\end{thm}
\begin{proof}
Let $A$ be a free basis of $F$. We may assume that $||\mu||_A=1$. By
Remark~\ref{rem:norm}, it suffices to prove that whenever $w_i$ is a
sequence of reduced cyclic words in $F(A)$ such that
$\mu=\lim_{n\to\infty}\frac{\eta_{w_i}}{||w_i||_A}$ then
$\lim_{i\to\infty} \frac{||w_i||_T}{||w_i||_A}=0$. Let $w_i$ be such
a sequence of cyclic words. Note that
$\frac{||w_i^t||_T}{||w_i^t||_A}=\frac{||w_i||_T}{||w_i||_A}$ and
$\frac{\eta_{w_i^t}}{||w_i^t||_A}=\frac{\eta_{w_i}}{||w_i||_A}$ for
any integer $t\ge 1$. Thus, by taking powers if necessary, we may
assume that $\lim_{i\to\infty}||w_i||_A=\infty$.

Let $p\in T$.  Let $C_1>0$ be such that $BBT_{T,p}(A)<C_1$. Recall
that such $C_1$ exists by Proposition~\ref{prop:BBT}. Also, we put
$C_2=\max_{a\in A} d_T(p,ap)$.

Let $\epsilon>0$ be arbitrary. Let $N\ge 1$ be such that
$\frac{1+10C_1}{N}\le \epsilon/2$. Choose $\epsilon_1>0$ such that
$C_2N\epsilon_1\le \epsilon/2$.  Put $w_i'=w_i^N$. Hence
$||w_i'||_A=N||w_i||_A$ is divisible by $N$ for every $i\ge 1$.
Then
$\frac{\eta_{w_i'}}{||w_i'||_A}=\frac{\eta_{w_i}}{||w_i||_A}$ and hence
\[
\mu=\lim_{i\to\infty} \frac{\eta_{w_i'}}{||w_i'||_A}.
\]

Write $w_i'$ as a cyclic concatenation $u_i=y_1\dots y_m$ where
$m=||w_i'||_A/N$ and where each $y_j\in F(A)$ has $|y_j|_A=N$. We say
that $y_j$ is \emph{good} if $y_j\in supp_A(\mu)$ and that $y_j$ is
\emph{bad} otherwise.  Write $m=m_{good}+m_{bad}$ where $m_{good}$
is the number of those $j=1,\dots, m$ for which $y_j$ is good. Since
$N$ is fixed and $\lim_{i\to\infty}
\frac{\eta_{w_i'}}{||w_i'||_A}=\mu$, for any fixed freely reduced word
$v\in F(A)$ of length $N$, the symmetrized frequencies $\frac{\langle v^{\pm 1}, w_i'\rangle_A}{||w_i'||_A}$
of $v$ in $w_i'$ converge to $\langle v,\mu\rangle_A$ as
$i\to\infty$. In particular, if $v\not\in supp_A(\mu)$, these
frequencies converge to zero. Therefore there exists some $i_0\ge 1$ such
that for every $i\ge i_0$ we have $m_{bad}\le \epsilon_1 ||w_i'||_A$.
Suppose now that $i\ge i_0$ is arbitrary.

Let $z_i\in F(A)$ be the freely reduced word obtained as a
(non-cyclic) concatenation $z_i=y_1\dots y_m$.

Suppose first that $y_j$ is good. Since $supp(\mu)\subseteq L^2
(T)$, there exists a cyclically reduced word $u_j\in F(A)$
containing $y_j$ as an initial segment such that $||u_j||_T\le 1$.
Then by Lemma~\ref{lem:LL} we have $d_T(p, u_jp)\le 1+2C_1$. Write
$u_j$ as a reduced product $u_j=y_ju_j'$. Then Lemma~\ref{lem:LL}
implies that
\[
d_T(p,y_jp)+d_T(p,u_j'p)\le d_T(p,u_jp)+2C_1\le 1+4C_1
\]
and hence
\[
d_T(p,y_jp)\le 1+4C_1.
\]

Suppose that $y_j$ is bad. Then obviously $d_T(p,y_jp)\le C_2N$.

Recall that $m_{bad}\le \epsilon_1 ||w_i'||_A$. Recall also that
$z_i=y_1\dots y_m\in F(A)$. Hence by Lemma~\ref{lem:LL}
\[
||w_i'||_T=||z_i||_T\le d_T(p,z_ip)\le m(1+4C_1)+\epsilon_1 ||w_i'||_A
C_2N.
\]

Then for every $i\ge i_0$
\[
\frac{||w_i'||_T}{||w_i'||_A}\le \frac{1+10C_1}{N}+ C_2N\epsilon_1\le
\epsilon,
\]
and hence
\[
0\le \langle T, \mu\rangle=\lim_{i\to\infty}
\frac{||w_i'||_T}{||w_i'||_A}\le \epsilon.
\]

Since $\epsilon>0$ was arbitrary, this implies that
$\langle T, \mu\rangle=0$, as required.
\end{proof}

\section{The dense orbits case}

\begin{prop}\label{prop:dense1}
Let $F$ be a finitely generated nonabelian free group. Let  $T$ be
an $\mathbb R$-tree with a very small minimal isometric action of
$F$  such that this action has dense orbits.

Let $\mu\in Curr(F)$ be such that $\langle T,\mu\rangle=0$.
Then $supp(\mu)\subseteq L^2(T)$.
\end{prop}

\begin{proof}
Choose a free basis $A$ of $F$. By re-scaling we may assume that
$||\mu||_A=1$. Then there exists a sequence of reduced cyclic words $w_i$ in $F(A)$
such that $\lim_{i\to\infty} \frac{\eta_{w_i}}{||w_i||}=\mu$. Since $\langle T,\mu\rangle=0$, we have
$\lim_{i\to\infty} \frac{||w_i||_T}{||w_i||_A}=0$.

Again, recall that
$\frac{||w_i^t||_T}{||w_i^t||_A}=\frac{||w_i||_T}{||w_i||_A}$ and
$\frac{\eta_{w_i^t}}{||w_i^t||_A}=\frac{\eta_{w_i}}{||w_i||_A}$ for
any integer $t\ge 1$. Thus, by taking powers if necessary, we may
assume that $\lim_{i\to\infty}||w_i||_A=\infty$.

Let $v\in F(A)$ be a freely reduced word with $\langle
v,\mu\rangle_A=\alpha>0$. We need to prove that for any $\epsilon>0$
there exists a cyclically reduced word $w$ in $F(A)$ containing $v$
as a subword and such that $||w||_T\le \epsilon$.

Since $||\mu||_A=1$, there exists a sequence of cyclic word $w_i$ in
$F(A)$ such that
\[
\lim_{i\to\infty}\frac{\eta_{w_i}}{||w_i||_A}=\mu
\]
and such that
\[
\lim_{i\to\infty} \frac{||w_i||_T}{||w_i||_A}=0.
\]

 Hence
 \[
\lim_{i\to\infty}\frac{\langle v^{\pm 1}, w_i\rangle}{||w_i||} =
\langle v,\mu\rangle_A=\alpha>0.
\]
Choose $i_0\ge 1$ such that for every $i\ge i_0$ we have
\[
\alpha/2\le \frac{\langle v^{\pm 1}, w_i\rangle}{||w_i||} \le
2\alpha.
\]

Let $\epsilon>0$ be arbitrary. By Proposition~\ref{prop:LL} there is
some free basis $B$ of $F$ and a point $p\in T$ such that
$BBT_{T,p}(B) <\epsilon$ and $d_T(p,bp)<\epsilon$ for every $b\in B$.
Let $n=n_{A,B}\ge 1$ be the double bounded cancellation constant for
$A,B$ and let $C=C_{A,B}\ge 1$ be the Lipshitz constant for $A,B$.

Let $i\ge i_0$.  Choose a maximal collection $v_1,\dots, v_m$ of
occurrences of $v^{\pm 1}$ in $w_i$ such that the $n$-neighborhoods
of these occurrences in $w_i$ do not overlap. Note that if an
occurrence of $v^{\pm 1}$ begins outside of the union of the
$n$-neighborhoods of $v_1,\dots, v_m$ in $w_i$, then, by maximality,
the beginning of that occurrence is at most $2n+|v|$ away
from the union of the $n$-neighborhoods of $v_1,\dots, v_m$. This
implies that the total number of occurrences of $v^{\pm 1}$ in $w_i$
satisfies:

\[
\langle v^{\pm 1}, w_i\rangle \le 2m(2n+|v|).
\]
Hence
\[
\frac{m}{||w_i||_A}\ge \frac{\langle v^{\pm 1},
w_i\rangle}{||w_i||_A (4n+2|v|)}\ge
\frac{\alpha}{2}\frac{1}{4n+2|v|}.
\]
Thus:
\[
\frac{\alpha}{K}\le \frac{m}{||w_i||}\le 2\alpha,
\]
where $K=8n+4|v|$.

Recall that $\lim_{i\to\infty} \frac{||w_i||_T}{||w_i||_A}=0$. Let
$\epsilon_1>0$ be such that $\epsilon_1 \frac{K}{\alpha}<\epsilon$.
We may assume that $i\ge i_0$ was chosen big enough so that
$\frac{||w_i||_T}{||w_i||_A}\le \epsilon_1$.

We now subdivide the cyclic word $w_i$ as a cyclic concatenation
$w=y_1\dots y_m$ where each segment $y_j\in F(A)$ of $w_i$ contains
the $n$-neighborhood of $v_j$ in $w_i$ for $j=1,\dots, m$. For
$j=1,\dots, m$ let $y_j'$ be the freely reduced word in $F(B)$
representing $y_j$.

Let $z_j$ be the portion of $y_j'$ that survives after the free
reductions in $y_{j-1}'y_j'y_{j+1}'$. Note that $z_j$ is nonempty by
definition of $n=n_{A,B}$. Then the cyclic concatenation
$w_i'=y_1'\dots y_m'$ is the cyclic word in $F(B)$ representing the
same conjugacy class as $w_i$. Note that the words $y_j'$ need not be
cyclically reduced. We choose letters $b_j\in B^{\pm 1}$ such that
each $x_j:=b_jz_j$ is cyclically reduced over $B$ and such that
$x_1\dots x_m$ is a freely reduced and cyclically reduced word over
$B$. Let $w_i''$ be the cyclic word obtained by cyclic concatenation
$w_i''=x_1\dots x_m$.

By Lemma~\ref{lem:LL} we have
\[
\left| d_T(p,y_1'\dots y_m' p)-\sum_{j=1}^m d_T(p,z_jp) \right|\le
4m\epsilon,
\]
\[
\big|  d(p,x_jp)-d(p,z_jp)\big| \le \epsilon   , \text{ and }
\]
\[
\left| d_T(p,x_1\dots x_m p)-\sum_{j=1}^m d_T(p,z_jp) \right|\le
4m\epsilon.
\]
Hence
\[
\left| d_T(p,y_1'\dots y_m' p)-d_T(p,x_1\dots x_m p)\right|\le
12m\epsilon.
\]
Since both $y_1'\dots y_m'$ and $x_1\dots x_m$ are cyclically
reduced, Lemma~\ref{lem:LL} also implies
\begin{gather*}
\left| d_T(p,y_1'\dots y_m' p)-||w_i'||_T\right|\le 2\epsilon,\\
\left| d_T(p,x_1\dots x_m p)-||w_i''||_T\right|\le 2\epsilon, \text{
and hence }\\
\left| ||w_i'||_T-||w_i''||_T \right|\le 12m\epsilon+4\epsilon.
\end{gather*}

Thus
\[
||w_i''||_T\le ||w_i'||_T+12m\epsilon+4\epsilon.
\]
 Since $x_1,\dots, x_m$ are cyclically reduced, Lemma~\ref{lem:LL}
 again implies that
 \[
\sum_{j=1}^m||x_j||_T\le ||w_i''||_T+4m\epsilon\le
||w_i'||_T+16m\epsilon+4\epsilon.
\]
Therefore there exists $j$ such that \[||x_j||_T\le
\frac{||w_i'||_T}{m}+16\epsilon+\frac{4\epsilon}{m}\le
\frac{||w_i'||_T}{m}+20\epsilon.\] Recall that $w_i$ and $w_i'$
represent the same conjugacy class in $F$ and so
$||w_i||_T=||w_i'||_T$. Thus
\[
||x_j||_T\le
\frac{||w_i||_T}{m}+20\epsilon=\frac{||w_i||_T}{||w_i||_A}
\frac{||w_i||_A}{m}+20\epsilon\le \epsilon_1\frac{K}{\alpha}
+20\epsilon\le 21\epsilon
\]
where the last inequality holds by the choice of $\epsilon_1$.

We now rewrite $x_j=b_jz_j$ as a freely reduced word in $F(A)$ and
then cyclically reduce the result to get a cyclically reduced word
$s_j$ in $F(A)$. By definition of $n=n_{A,B}$, the construction of
$x_j$ implies that the occurrence $v_j$ of $v^{\pm 1}$ in $y_j$
survives intact in $s_j$. Thus $s_j$ is a cyclically reduced word in
$F(A)$ containing $v^{\pm 1}$ as a subword and satisfying
$||s_j||_T\le 21\epsilon$. Since $\epsilon>0$ was arbitrary, this
implies that $v\in L^2(T)$, as required.
\end{proof}

\section{Approximating the word metric for a graph of groups}

\begin{conv}
If $G=\langle S\rangle$ is a group with a finite-generating set $S$,
we say that a cyclic word $w$ in $S^{\pm 1}$ is a \emph{cyclic
$\lambda$-quasigeodesic} with respect to $d_S$ if for every vertex
on $w$ cutting $w$ open at that vertex produced a word that is
$\lambda$-quasigeodesic. We will denote the Cayley graph of $X$ with
respect to $S$ by $X(G,S)$.
\end{conv}

\begin{lem}\label{lem:subst}
Let $G$ be a word-hyperbolic group with a fixed finite generating
set $S$ and let $\lambda>0$. Then there exists
$\lambda'=\lambda'(\lambda,S,G)>0$ with the following property.
Suppose that $w$ is a cyclic word in $S^{\pm 1}$ that is a cyclic
$\lambda$-quasigeodesic for $G$ with respect to $d_S$. Let $w$ be
subdivided as a cyclic concatenation $w=v_1\dots v_m$ where each
$v_i$ is a $\lambda$-quasigeodesic word in $S^{\pm 1}$. For
$i=1,\dots, m$ let $z_i$ be a another $\lambda$-quasigeodesic word
representing the same element of $G$ as $v_i$. Let $w'$ be the
cyclic word obtained by a cyclic concatenation $w'=z_1\dots z_m$.
Then $w'$ is a cyclic $\lambda'$-quasigeodesic.
\end{lem}

\begin{proof}
The proof is a straightforward variation on the proof of a similar
Lemma~3.4 in \cite{Ka97} for non-cyclic word and we leave the
details to the reader. Note, however, that the proof uses the fact
that $G$ is word-hyperbolic and hence the paths labelled by $z_i$
and $v_i$ in the Cayley graph of $G$ are $\epsilon$-Hausdorff close,
where $\epsilon=\epsilon(\lambda,G,S)>0$. It is not hard to see that
the statement of this lemma in general fails for non-hyperbolic
groups, e.g. for $\mathbb Z^2$.
\end{proof}

\begin{notation}[Graph of groups notations]
Let $\mathcal Y$ be a graph of groups~\cite{Bass,Serre}. For $x\in
VY$ we will denote the vertex group of $x$ by $G_x$ and for $e\in
EY$ we will denote the edge group of $e$ by $G_e$. For $e\in EY$ we
denote the initial vertex of $e$ by $o(e)$ and we denote the
terminal vertex of $e$ by $t(e)$. For $e\in EY$ we denote the
corresponding boundary monomorphisms by $\alpha_e: G_e\to G_{o(e)}$
and $\omega_e: G_e\to G_{o(e)}$. Recall that according to the
standard graph of groups conventions, for every $e\in EY$ we have
$G_e=G_{e^{-1}}$, $o(e)=t(e^{-1})$ and $\alpha_e=\omega_{e^{-1}}$.
\end{notation}

\begin{defn}\label{defn:y-paths}
Let $\mathcal Y$ is a finite connected reduced graph of groups with
underlying graph $Y$, where all the vertex groups $G_x, x\in VY$,
are finitely generated. Let $Z\subseteq Y$ be a maximal tree in $Y$.
Fix an orientation $EY=E^+Y\sqcup E^-Y$ on $Y$, so that for every
$e\in EY$ we have $e\in E^+Y$ iff $e^{-1}\in E^-Y$.

Let $G=\pi_1(\mathcal Y,Z)$. Then $G$ has the presentation

\begin{gather*} G=\left( \big(\bigstar_{x\in VY} G_x\big) \ast
F(E^+Y)\right)/\\ \langle\langle e^{-1}\alpha_e(c)e=\omega_e(c),
e\in E^+Y, c\in G_e; e=1, e\in E^+Z \rangle\rangle.
\end{gather*}

Let $S_x$ be a finite generating set of $G_x$ for $x\in VY$. Then

\[
S=S_{\mathcal Y}=E^+Y\bigcup \cup_{x\in VY} S_x
\]
is a generating set of $G$ that is said to be \emph{adapted to
$\mathcal Y$}.

Recall that a \emph{$\mathcal Y$-path} from $x\in VY$ to $x'\in VY$
is a sequence
\[
\alpha=g_0, e_1, g_1, e_2, \dots, e_n, g_n
\]
where $e_1,\dots, e_n$ is an edge-path in $Y$ with the vertex
sequence $x=x_0, x_1,\dots, x_n=x'$ and where $g_i\in G_{x_i}$. A
$\mathcal Y$-path $\alpha$ as above is \emph{$\mathcal Y$-reduced}
if it does not contain a subsequence of the form
\[
e, \omega_e(c), e^{-1},
\]
where $e\in EY$ and $c\in G_e$.

We say that a word $W$ in $S_{\mathcal Y}^{\pm 1}$ is \emph{a
path-word for $\mathcal Y$} if it has the form
\[
W=w_0e_1w_1\dots e_nw_n
\]
where $e_1,\dots, e_n$ is an edge-path in $Y$ with the vertex
sequence $x_0, x_1,\dots, x_n$ and where $w_i$ is a word in
$S_{x_i}^{\pm 1}$. Note that any path-word $W$ as above defines a
$\mathcal Y$-path $\alpha=g_0, e_1, g_1, e_2, \dots, e_n, g_n$,
where $g_i\in G_{x_i}$ is the group element represented by $w_i$.

We say that a path-word $W$ as above is a \emph{$\mathcal Y$-reduced
path word} if it defines a $\mathcal Y$-reduced $\mathcal Y$-path.

In a similar way, one defines the notions of a \emph{cyclic
$\mathcal Y$-path}, \emph{$\mathcal Y$-reduced cyclic $\mathcal
Y$-path}, a \emph{cyclic path-word for $\mathcal Y$} and a
\emph{$\mathcal Y$-reduced cyclic path-word for $\mathcal Y$}.
\end{defn}

\begin{conv}\label{conv:subst}
Let $A$ be another finite generating set for $G=\pi_1(\mathcal Y,
Z)$. For each element of $S_{\mathcal Y}$ choose its representation
as a word in $A$. Using these representations for substitution, any
(cyclic) word $W$ in $S_{\mathcal Y}^{\pm 1}$ defines a (cyclic)
word $\widetilde W$ in $A^{\pm 1}$.
\end{conv}

\begin{prop}\label{prop:cyclic}

Let $G$ be a word-hyperbolic group and let $G=\pi_1(\mathcal Y, Z)$
be a splitting of $G$ where $\mathcal Y$ is a finite graph of
groups, and where $Z\subseteq Y$ is a maximal subtree. Suppose that
each vertex group $G_x$ is quasiconvex in $G$. Let $S_x$ be a finite
generating set of $G_x$ for each $x\in VY$ and let
\[
S_{\mathcal Y}=E^+Y\bigcup \cup_{x\in VY} S_x
\]
be a generating set of $G$ adapted to $\mathcal Y$. Let $A$ be any
other finite generating set of $G$.

There exist $\lambda>0$ and $\epsilon>0$ with the following
properties.

\begin{enumerate}
\item For any $g\in F, g\ne 1$ there exists a $\mathcal Y$-reduced
  path-word $W$ representing $g$ such that $W$ is
  $\lambda$-quasigeodesic with respect to $d_{S_{\mathcal Y}}$.

Moreover, if $w\in (A\cup A^{-1})^\ast$ is a $d_A$-geodesic
representing $g$, then the paths from $1$ to $g$ labelled by $w$ and
$\tilde W$ in $X(F,A)$ are $\epsilon$-Hausdorff close.

\item For any conjugacy class $[g], g\in F, g\ne 1$, there exists a $\mathcal Y$-reduced cyclic
path-word representing $[g]$ such that $W$ is a cyclic
$\lambda$-quasigeodesic with respect to $d_{S_{\mathcal Y}}$.

Moreover, let $w$ be a cyclic word over $A^{\pm 1}$ which is a
cyclic $d_A$-geodesic representing $[g]$. Then there is a function
$f$ from the vertex set of $w$ to the vertex set of $W$ with the
following properties. For every vertex $x$ on $w$ cutting open $w$
at $x$ and $W$ at $f(x)$ produces a $d_A$-geodesic word $w_x$ and a
$\mathcal Y$-reduced path-word $W_x$ accordingly. Additionally,
there exist paths $p$ and $p'$ in $X(G,A)$ labelled by $w_x$ and
$\widetilde{W_x}$ such that there initial vertices are at most
$\epsilon$-apart, their terminal vertices are at most
$\epsilon$-apart and, for any other vertex $y$ on $w$ the vertices
on $p$ and $p'$ corresponding to $y$ and $f(y)$ are at most
$\epsilon$-apart in $X(G,A)$.
\end{enumerate}

\end{prop}

\begin{proof}
Part (1) of the lemma is essentially the same as Proposition~5.1 in
\cite{Ka01}. We will briefly indicate the proof of part (2) that is
similar to to proof of Proposition~5.1 in \cite{Ka01}.

To see that part (2) holds, we first replace $g$ by a
$d_{S_{\mathcal Y}}$-shortest element $g'\in [g]$. Let $\hat W$ be a
$d_{S_{\mathcal Y}}$-geodesic representative of $g'$. Then, by the
choice of $g'$, every cyclic permutation of $\hat W$ is a
$d_{S_{\mathcal Y}}$-geodesic. Recall that $Z$ is a maximal tree in
$Y$.

Let $W_0$ be the cyclic word defined by $\hat W$. Thus $W_0$ is a
cyclic $d_{S_{\mathcal Y}}$-geodesic. By inserting in $W_0$ several
subwords in the alphabet $EZ$ (which therefore represent the trivial
element of $G$) of length $\le \#EZ$ each, we can obtain a new
cyclic word $W_1$ that is a cyclic path-word for $\mathcal Y$. By
Lemma~\ref{lem:subst} $W_1$ is a cyclic $\lambda_1$-quasigeodesic
for some constant $\lambda_1>0$. However, $W_1$ need not be
$\mathcal Y$-reduced.

We say that a subsegment $v$ of $W_1$ is a \emph{pinch} if $v$ has
the form $v=eue^{-1}$ and $v$ represents an element of the edge
group $\omega_e(G_e)$. Note that we DO NOT require $u$ to be a word
in the generating set of the vertex group $G_{t(e)}$. We say that a
collection of pinches $v_1,\dots, v_m$ in $W_1$ is \emph{separated}
if the segments $v_1,\dots, v_m$ do not overlap in $W_1$.

Let each of $v_1,\dots, v_m$ and $v_1',\dots, v'_l$ be a separated
collection of pinches in $W_1$. We say that $v_1,\dots, v_m$ is
\emph{dominated by} $v_1',\dots, v_l'$ if for each $v_i$ there is
some $v_j'$ such that the segment $v_i$ is contained in the segment
$v_j$ in $W_1$.

Now take $v_1=e_1u_1e_1^{-1},\dots, v_{m}=e_mu_me_m^{-1}$ be a
maximal (with respect to dominance) separated collection of pinches
in $W_1$. Replace each $v_i$ in $W_1$ by its geodesic representative
in the generators of the vertex group $G_{o(e_i)}$. Denote the
resulting cyclic word by $W$.

Since the vertex groups are quasiconvex in $G$,
Lemma~\ref{lem:subst} implies that the cyclic word $w$ is a cyclic
$\lambda_2$-quasigeodesic for some constant $\lambda_2>0$. The
maximality of the choice of $v_1,\dots, v_m$ implies that $W$ is a
$\mathcal Y$-reduced cyclic path-word.

Establishing the remaining properties of $W$ asserted in part (2) of
the proposition is a straightforward $\delta$-hyperbolic exercise
and we leave the details to the reader.
\end{proof}

\section{The discrete action case}

\begin{conv}\label{conv:F}
For the remainder of this section, let $F$ be a finitely generated
nonabelian free group. We fix a very small splitting
$F=\pi_1(\mathcal Y,Z)$, where $\mathcal Y$ is a finite connected
reduced graph of groups with a maximal tree $Z$, where all edge
groups are cyclic (and hence all vertex groups are finitely
generated free groups). We fix an orientation on $Y$ and choose a
finite generating set
\[
S_{\mathcal Y}=E^+Y\bigcup \cup_{x\in VY} S_x
\]
for $F$ that is adapted to $\mathcal Y$.

Let $V'Y$ denote the set of all $x\in VY$ such that the  vertex
group $G_x$ is non-cyclic. Let $T_\mathcal Y$ be the Bass-Serre tree
corresponding to the splitting $F=\pi_1(\mathcal Y,Z)$.

We also fix a free basis $A$ for $F$.
\end{conv}

\begin{lem}\label{lem:L2}
Let $L=L^2(T_{\mathcal Y})$ and let $v\in F(A)$ be freely reduced.
Then $v\in L_A$ if and only if for some vertex $x\in V'Y$ the word
$v$ can be read as the label of a path in the core of the Stallings
subgroup graph for $G_x$ with respect to $F(A)$. That is,
$L=\cup_{x\in V'Y} i_\Lambda(\partial^2 G_x)$
\end{lem}

\begin{proof}

Note first that, since $\mathcal Y$ defines a very small splitting
of $F$, every cyclic vertex group of $\mathcal Y$ is contained in
some non-cyclic vertex group. Hence every word in $F(A)$ readable in
the core of the Stallings subgroup graph of a cyclic vertex group is
also readable in the core of the Stallings subgroup graph of a
non-cyclic vertex group.

Since the action of $F$ on $T_{\mathcal Y}$ is discrete, there is
some $c>0$ such that for every $f\in F$ either $||f||_{T_\mathcal
Y}\ge c$ or else $f$ is conjugate to an element of some $G_x,x\in
VY$ and $||f||_{T_\mathcal Y}=0$. Therefore, in view of the above
remark, by definition of $L=L^2(T_{\mathcal Y})$, for $v\in F(A)$ we
have $v\in L_A$ if and only if $v$ is a subword of some cyclically
reduced word in $F(A)$ representing an element conjugate to an
element of some $G_x,x\in V'Y$. This implies the statement of the
lemma.
\end{proof}

The following lemma is an easy corollary of the definitions:
\begin{lem}\label{lem:core}
Let $H\le F(A)$ be a finitely generated subgroup. Let $M$ be the
number of edges in the Stallings subgroup graph $\Gamma_H$ of $H$
with respect to $F(A)$. Let $z_1,w,z_2$ be freely reduced words in
$F(A)$ such that $|z_1|,|z_2|\le c$ and such that $z_1\cdot w\cdot
z_2\in H$ (we do not assume this product to be freely reduced). Let
$w$ be written as a reduced product $w=w_1w'w_2$ where $|w_1|_A,
|w_2|_A\ge c+M$. Then $w'$ can be read as the label of a path in the
core of the graph $\Gamma_H$.
\end{lem}

\begin{prop}\label{prop:discr1}
Let $F$ be as in Convention~\ref{conv:F}. Let $T=T_{\mathcal Y}$ be
the Bass-Serre tree corresponding to $\mathcal Y$. Let $\mu\in
Curr(F)$ be such that $\langle T, \mu\rangle=0$.

Then $supp(\mu)\subseteq L^2(T)$.
\end{prop}

\begin{proof}
Choose a free basis $A$ of $F$. By re-scaling we may assume that
$||\mu||_A=1$. Let $w_i$ be a sequence of reduced cyclic words $w_i$
in $F(A)$ such that $\lim_{i\to\infty}
\frac{\eta_{w_i}}{||w_i||}=\mu$. Since $\langle T, \mu\rangle=0$, we
have  $\lim_{i\to\infty} \frac{||w_i||_T}{||w_i||_A}=0$. As before, by taking powers if
necessary, we may assume that $\lim_{i\to\infty}||w_i||_A=\infty$.
We need to show that $supp(\mu)\subseteq L^2(T)$.

Let $\lambda,\epsilon>0$ be as provided by
Proposition~\ref{prop:cyclic}. Let $M$ be the maximum of the numbers
of edges among the Stallings subgroups graphs with respect to $F(A)$
for the subgroups $G_x, x\in VY$.

Let $v\in F(A)$ be a freely reduced word such that $v\in
supp_A(\mu)$. Thus $\langle v, \mu\rangle_A >0$. Then there exists a
freely reduced word $v'=u_1vu_2$ such that $\alpha:=\langle v',
\mu\rangle_A >0$ and such that $|u_1|_A, |u_2|_A\ge M+\epsilon$. We
may assume that for every $i\ge 1$ we have
\[
\alpha/2\le \frac{\langle (v')^{\pm 1}, w_i\rangle}{||w_i||_A}\le
2\alpha.
\]

Let $W_i$ be a cyclically $\lambda$-quasigeodesic $\mathcal
Y$-reduced cyclic path-word representing the same conjugacy class in
$F$ as $w_i$.

Fix $N>>\epsilon$.

By the same argument as in the proof of
Proposition~\ref{prop:dense1} we may find $m$ distinct occurrences
$v_1,\dots, v_m$ of $(v')^{\pm 1}$ in $w_i$ such that
$N$-neighborhoods of $v_1,\dots, v_m$ do not overlap in $w_i$ and
such that

\[
\frac{\alpha}{K}\le \frac{m}{||w_i||_A}\le 2\alpha,
\]
where $K=8N+4|v'|$.

For each $j=1,\dots, m$ let $x_j,x_j'$ be the initial and terminal
vertices of $v_j$ in $w_i$. Let $V_j$ be the segment of $W_i$ from
$f(x_j)$ to $f(x_j')$. The fact that the occurrences $v_1,\dots,
v_m$ of $(v')^{\pm 1}$ in $w$ have non-overlapping $N$-neighborhoods
implies that the segments $V_1,\dots, V_m$ in $W_i$ have no
overlaps.

Suppose first that for every $w_i$ each $V_j$ contains an occurrence
of some $e\in EY$. Then $||w_i||_T=||W_i||_T\ge m$. Hence for every
$i\ge 1$

\[
\frac{||w_i||_T}{||w_i||_A}\ge \frac{m}{||w_i||_A}\ge
\frac{\alpha}{K}>0.
\]
This contradicts the assumption that $\lim_{i\to\infty}
\frac{||w_i||_T}{||w_i||_A}=0$. Thus there is some $w_i$ such that
some $V_j$ contains no occurrences of $EY$. Since $W_i$ is a cyclic
path-word for $\mathcal Y$, this implies that there is some vertex
$x$ of $Y$ such that $V_j$ is a word in $S_x^{\pm 1}$.

Proposition~\ref{prop:cyclic} implies that in $X(F,A)$ there are
paths $\alpha_1, \alpha_2$ labelled by $v'$ and a word in $S_x^{\pm
1}$ accordingly, such that the initial vertices of
$\alpha_1,\alpha_2$ are at most $\epsilon$-apart and the terminal
vertices of $\alpha_1,\alpha_2$ are at most $\epsilon$-apart. Recall
that $v'=u_2vu_2$ where $|u_1|, |u_2|\ge M+\epsilon$. Therefore by
Lemma~\ref{lem:core} $v$ can be read as a path in the Stallings core
graph of $G_x=\langle S_x\rangle\le F(A)$. Hence by
Lemma~\ref{lem:L2} $v\in L^2(T)_A$, as required.
\end{proof}

\section{Restricting geodesic currents to subgroups}

\begin{defn}[Restriction of a current to a subgroup]
Let $H\le F$ be a finitely generated nonabelian subgroup. Since $H$
is quasiconvex in $F$, we have a canonical $H$-equivariant
topological embedding $\partial H\subset \partial F$ which induces a
canonical $H$-equivariant  topological embedding $\partial^2
H\subset
\partial^2 F$. Let $\mu\in Curr(F)$. We define the \emph{restricted
current} $\mu|_H$ as follows. For any Borel subset $S\subseteq
\partial^2H$ put $\mu|_H(S):=\mu(S)$. It is easy to see that $\mu|_H$
is an $H$-invariant measure on $\partial^2 H$, that is $\mu|_H\in
Curr(H)$.

Denote $r_H: Curr(F)\to Curr(H)$, $r_H(\mu)=\mu|_H$ for $\mu\in
Curr(F)$.
\end{defn}

\begin{rem}
In the above definition, the restricted current $\mu|_H$ depends
only on $\mu$ and the conjugacy class of $H$ in $F$ in the
following sense. Let $f\in F$ and put $H_1=fHf^{-1}$. Consider an
isomorphism $\alpha:H\to H_1$ defined as $\alpha(h)=fhf^{-1}$ for
$h\in H$. Then $\alpha$ induces an $\alpha$-equivariant
homeomorphism $\widehat \alpha:
\partial^2 H\to\partial^2 H_1$. Hence $\widehat \alpha$ induces a
canonical isomorphism $\breve\alpha: Curr(H)\to Curr(H_1)$ defined
as $(\breve\alpha\nu)(S_1)=\nu ({\widehat\alpha}^{-1}(S_1))$ for any
$\nu\in Curr(H)$ and $S_1\subseteq
\partial^2H_1$.

It is not hard to check that for any  $\mu\in Curr(F)$ we have
$\breve\alpha(\mu|_H)=\mu|_{H_1}$.
\end{rem}

The following is an immediate corollary of the definitions.

\begin{prop}
let $H\le F$ be a finitely generated nonabelian subgroup. Then $r_H:
Curr(F)\to Curr(H)$ is a linear continuous map.
\end{prop}

\begin{notation}
Let $G$ be a group and $g\in G$ be an element. We denote by $[g]_G$
the conjugacy class of $g$ in $G$.

\end{notation}

The proof of the following lemma given below was suggested to the
authors by Gilbert Levitt.

\begin{lem}
Let $H\le F$ be a finitely generated subgroup.

\begin{enumerate}
\item There exists an integer $n\ge 1$ with the following property.
For every $h\in H$ there exist $h_1,\dots, h_m\in H$ with $m\le n$
such that
\[
[h]_F\cap H=[h_1]_H\cup [h_2]_H\cup \dots \cup [h_m]_H.
\]

We denote by $n_{H,F}$ the smallest $n\ge 1$ with this property.
\item There exists an integer $n_1\ge 1$ with the following
property. Let $h\in H$ be any nontrivial element that is not a
proper power in $H$. Represent $h$ as $h=f^d$ where $f\in F$ is not
a power and $d\ge 1$. Then $d\le n_1$.

We denote by $d_{H,F}$ the smallest $n_1\ge 1$ with this property.

\item A finitely generated subgroup $H\le F$ is malnormal if and
only if $n_{H,F}=d_{H,F}=1$.

\end{enumerate}

\end{lem}

\begin{proof}
It is easy to see that if $H\le F$ is malnormal then
$n_{H,F}=d_{H,F}=1$ satisfy the requirements of the lemma.
Similarly, if $n_{H,F}=d_{H,F}=1$, it is not hard to show that $H\le
F$ is malnormal. Thus part (3) of the lemma holds.

Suppose now that $H\le F$ is an arbitrary finitely generated
subgroup. By Marshall Hall's Theorem there exist subgroups $K,U\le
F$ such that $U=\langle H,K\rangle =H\ast K$ has finite index in
$F$. Let $p=[F:U]$. Choose $f_1,\dots, f_p\in F$ such that
$F=\cup_{j=1}^p Uf_j$.

Let $h\in H, h\ne 1$ be arbitrary. Let $J=\{j| 1\le j\le p,
f_{j}^{-1}hf_j\in U\}$. Every element of $F$ has the form $uf_j$,
where $u\in U$, $1\le j\le p$. Hence it is obvious that
\[
[h]_F\cap U=\cup_{j\in J} [f_j^{-1}hf_j]_U
\]
Let $J'=\{j\in J: f_j^{-1}hf_j\in H\}$. Since $H\le U=H\ast K$ is
malnormal in $U$, it follows that for $j\in J'$ we have
\[
[f_j^{-1}hf_j]_U\cap H=[f_j^{-1}hf_j]_H
\]
and for $j\in J-J'$ we have
\[
[f_j^{-1}hf_j]_U\cap H=\varnothing.
\]

Therefore
\[
[h]_F\cap H=\cup_{j\in J'}[f_j^{-1}hf_j]_H.
\]
Since $h\in H, h\ne 1$ was arbitrary, we see that $n_{H,F}\le
p=[F:U]$ and part (1) of the lemma is proved.

Now let $U_1\le U$ be such that $U_1\le F$ is normal and of finite
index. Let $q=[F:U_1]$. It is obvious that for any $f\in F$ we have
$f^q\in U_1$ and thus $f^q\in U$.

Now let $h\in H, h\ne 1$ be an element that is not a proper power in
$H$. Let $h=f^d$ where $f\in F$, $d\ge 1$ and where $f\in F$ is not
a proper power. Note that since $H\le U=U\ast K$ is malnormal in
$U$, it follows that $h$ is not a proper power in $U$. Let $t\ge 1$
be the smallest such that $f^t\in U$. Note that, by the above
remark, $t\le q=[F:U_1]$. Note also, that, since $H$ is a free
factor of $U$ and since $(f^t)^d=h^t\in H$, it follows that $f^t\in
H$.

Suppose that $d>t$. If $d$ is divisible by $t$, $d=ts$, then $f^t\in
U$ and $f^{ts}\in H$, where $s>1$. Since $f^t\in H$ and $h=(f^t)^s$,
this contradicts our assumption that $h\in H$ is not a proper power
in $H$. If $d$ is not divisible by $t$, we can write $d=ts+r$ where
$0<r<t$. Since $f^d=h\in U$, $f^t\in U$, it follows that $f^r\in U$.
Since $0<r<t$, this contradicts the choice of $t$. Thus $d\le t\le
q=[F:U_1]$. Hence $d_{H,F}\le [F:U_1]$ and part (2) of the lemma is
proved.

\end{proof}

The following lemma follows easily from the definitions of counting
currents in terms of delta-functions:

\begin{lem}\label{lem:split}
Let $H\le F$ be a nonabelian finitely generated subgroup. Let $h\in
H$ be a nontrivial element such that $h$ is not a proper power in
$H$. Let $h_1,\dots, h_m\in H$ be pairwise non-conjugate elements
such that $[h]_F\cap H=[h_1]_H\cup [h_2]_H\cup \dots \cup [h_m]_H$.
Let $h=f^d$ where $d\ge 1$ and $f\in F$ is not a proper power in
$F$.

Then:

\begin{enumerate}
\item We have $\eta_h^F|_H=d\sum_{i=1}^m \eta_{h_i}^H$.

\item If $H\le F$ is malnormal then $d=m=1$ and
$\eta_h^F|_H=\eta_h^H$.

\end{enumerate}

\end{lem}



\begin{lem}\label{lem:sgp}
Let $H\le F$ be a finitely generated subgroup.  Let $A$ be a free
basis of $F$ and let $\Gamma_H$ be the Stallings subgroup graph for
$H\le F$ with respect to $A$.

Let $\mu\in Curr(F)$. Then $supp(\mu)\subseteq i_\Lambda(\partial^2
H)$ if and only if for every $v\in F(A)$ with $\langle v,
\mu\rangle>0$ there is a reduced edge-path labelled by $v$ in
$Core(\Gamma_H)$. That is, $supp(\mu)\subseteq i_\Lambda(\partial^2
H)$ if and only if for each $v\in F(A)$ that cannot be read along a
reduced path in $Core(\Gamma_H)$ we have $\langle v,\mu\rangle=0$.

\end{lem}
\begin{proof}
Note that the convex hull $T$ of $\partial H$ in $X(F,A)$ is a
minimal $H$-invariant subtree in $X(F,A)$ and that $T$ can also be
seen as a copy of the universal cover of $Core(\Gamma_H)$ in
$X(F,A)$.

It is not hard to see that for $v\in F(A)$ exactly one of the
following two alternatives holds:

(a) Every segment labelled by $v$ in $X(F,A)$ is contained in some
$F$-translate of $T$ in $F(X,A)$ and every such segment defines a
two-sided cylinder that intersects nontrivially some $F$-translate
of $\partial^2 H$ in $\partial^2F$.

(b) No segment labelled by $v$ in $X(F,A)$ is contained in an
$F$-translate of $T$ and every such segment defines a two-sided
cylinder that is disjoint from every $F$-translate of $\partial^2H$
in $\partial^2F$.

This fact immediately implies the statement of the lemma.
\end{proof}

\begin{notation}
Let $H\le F$ be a finitely generated subgroup. We denote by
$Curr_H(F)$ the set of all $\mu\in Curr(F)$ such that
$supp(\mu)\subseteq i_\Lambda(\partial^2 H)$.
\end{notation}

Lemma~\ref{lem:sgp} immediately implies:

\begin{lem}
Let $H\le F$ be a finitely generated subgroup. Then $Curr_H(F)$ is a
closed affine subspace of $Curr(F)$.
\end{lem}

\begin{prop}\label{prop:appr}
Let $H\le F$ be a finitely generated subgroup and let $\mu\in
Curr_H (F)$.
Then there exist a sequence $h_i\in H$ and a sequence $\lambda_i\ge
0$ such that
\[
\mu=\lim_{i\to\infty} \lambda_i \eta^F_{h_i}.
\]
\end{prop}

\begin{proof}
Fix a free basis $A$ of $F$. By re-scaling we may assume that
$||\mu||_A=1$. Let $w_i$ be a sequence of reduced cyclic words in
$F(A)$ such that
$\lim_{i\to\infty}\frac{\eta_{w_i}}{||w_i||_A}=\mu$.

Let $\Gamma_H$ be the Stallings subgroup graph of $H$ with respect
to $A$. Let $C>0$ be such that for any directed edges $e,e'$ in
$Core(\Gamma_H)$ there exists an edge-path $p$ of length $\le C$ in
$Core(\Gamma_H)$ such that $epe'$ is a reduced edge-path in
$Core(\Gamma_H)$.

It suffices to prove that for every $M\ge 1$ and every $\epsilon>0$
there exists a cyclic path in $Core(\Gamma_H)$ labelled by a cyclic
word $w$ such that for every $v\in F(A)$ with $|v|_A=M$ we have
\[
\left| \frac{\langle v^{\pm 1}, w\rangle}{||w||_A}-\langle v,
\mu\rangle\right|\le \epsilon.\tag{$\ast\ast$}
\]

Let an integer $M\ge 1$ and a real number $\epsilon>0$ be arbitrary. Let $\alpha=\max_{|v|_A=M}\langle
v,\mu\rangle$.

Choose an integer $N>M\ge 1$ be such that

\[
(\alpha+\epsilon)(1-\frac{N}{N+C})\le \epsilon,
\]
and such that $\frac{M+C}{N}\le \epsilon/3$.

By replacing $w_i$ by their high powers if necessary, we may assume
that $||w_i||_A$ is divisible by $N$ for every $i\ge 1$ and that
$\displaystyle\lim_{i\to\infty} ||w_i||_A=\infty$.

Since $\mu=\lim_{i\to\infty} \frac{1}{||w_i||_A}\eta_{w_i}$, there exists
$i_0\ge 1$ such that for every $i\ge i_0$  and every $v\in F$
with $|v|_A=M$ we have
\[
\left| \frac{\langle v^{\pm 1}, w_i\rangle}{||w_i||_A}-\langle v,
\mu\rangle \right|\le \epsilon.
\]
Hence, by the choice of $N$, for any $v\in F$ with $|v|_A=M$ and for
any $0\le x\le \langle v^{\pm 1}, w_i\rangle$ we have
\[
\left|\frac{x}{||w_i||_A}-\frac{x}{||w_i||_A}\frac{N}{N+C}\right|\le
\epsilon.
\]

Let $\epsilon_1>0$ be such that $N\epsilon_1\le \epsilon/3$.

Let $i\ge i_0$. Put $m=||w_i||_A/N$ and write $w_i$ as a cyclic
concatenation $w_i=y_1\dots y_m$ of segments of length $N$. As
before, for $j=1,\dots, m$ we say that $y_j$ is \emph{good} if
$y_j\in supp(\mu)$ and say that $y_j$ is \emph{bad} otherwise. Let
$m=m_{bad}+m_{good}$ where $m_{bad}$ is the number of those
$j=1,\dots, m$ such that $y_j$ is bad.

Since $\lim_{i\to\infty}\frac{\eta_{w_i}}{||w_i||_A}=\mu$, for any fixed freely reduced word
$v\in F(A)$ of length $N$ the symmetrized frequencies $\frac{\langle v^{\pm 1}, w_i\rangle_A}{||w_i||_A}$
of $v$ in $w_i$ converge to $\langle v,\mu\rangle_A$ as
$i\to\infty$. In particular, if $v\not\in supp_A(\mu)$, these
frequencies converge to zero. It follows that there exists
$i_1\ge i_0$ such that for every $i\ge i_1$ we have $m_{bad}\le
\epsilon_1 ||w_i||_A$.

Let $i\ge i_1$. For each $j=1,\dots, m$ such that $y_j$ is good
choose a path $p_j$ in $Core(\Gamma_H)$ labelled by $y_j$ and put
$z_j=y_j$ in this case. For each $j$ such that $y_j$ is bad choose
any reduced path $p_j$ of length $N$ in $Core(\Gamma_H)$. Put $z_j$
to be the label of $p_j$ in this case.

Finally, choose reduced paths $p_1',\dots, p_m'$ in
$Core(\Gamma_H)$, of length $\le C$ each, such that
$p_1p_1'p_2p_2'\dots p_jp_j'$ defines a reduced cyclic path $\alpha$
in $Core(\Gamma_H)$. Let $r_j$ be the label of $p_j'$. Then the
label of the cyclic path $\alpha$ is the cyclic word $w_i'$ obtained
as a cyclic concatenation
\[
w_i'=z_1r_1\dots z_mr_m.
\]
Since $|z_j|=N$ and $|r_j|\le C$, we have
\[
||w_i||_A\le ||w_i'||_A\le ||w_i||_A+mC,
\]
and hence
\[
1\le \frac{||w_i'||_A}{||w_i||_A}\le 1+\frac{C}{N}.\tag{$\dag$}
\]

Note that by construction $w_i'$ represents the $F$-conjugacy class
of some $h_i\in H$.

Let $v\in F(A)$ be any freely reduced word of length $M$. We want to
compare the frequencies of $v$ in $w_i$ and $w_i'$. We say that an
occurrence of $v^{\pm 1}$ in $w_i$ is \emph{good} if it lies in some
good $y_j$ and that this occurrence is \emph{bad} otherwise. Let
$\langle v^{\pm 1}, w_i\rangle=\langle v^{\pm 1},
w_i\rangle_{good}+\langle v^{\pm 1}, w_i\rangle_{bad}$ where
$\langle v^{\pm 1}, w_i\rangle_{good}$ is the number of good
occurrences of $v$ in $w_i$. Note that
\[
\langle v^{\pm 1}, w_i\rangle_{bad}\le Nm_{bad}+M m\le
N\epsilon_1||w_i||_A+Mm.
\]

Hence
\[
\frac{\langle v^{\pm 1}, w_i\rangle_{bad}}{||w_i||_A}\le
N\epsilon_1+\frac{M}{N}\le \epsilon.
\]

Thus
\[
\left| \frac{\langle v^{\pm 1},
w_i\rangle_{good}}{||w_i||_A}-\frac{\langle v^{\pm 1},
w_i\rangle}{||w_i||_A}\right|\le \epsilon.
\]

Similarly, we say that an occurrence of $v^{\pm 1}$ in $w_i'$ is
\emph{good} if it lies within some $z_j=y_j$ where $y_j$ is good,
and that an occurrence of $v^{\pm 1}$ in $w_i'$ is \emph{bad}
otherwise.

We have
\[
\langle v^{\pm 1}, w_i'\rangle_{bad}\le Nm_{bad}+Mm+Cm\le
N\epsilon_1||w_i||_A+Mm+Cm.
\]
Hence
\[
\frac{\langle v^{\pm 1}, w_i'\rangle_{bad}}{||w_i||_A}\le
N\epsilon_1+\frac{M}{N}+\frac{C}{N}.
\]
Therefore
\[
\frac{\langle v^{\pm 1},
w_i'\rangle_{bad}}{||w_i'||_A}=\frac{\langle v^{\pm 1},
w_i\rangle_{bad}}{||w_i||_A}\frac{||w_i||_A}{||w_i'||_A}\le
N\epsilon_1+\frac{M}{N}+\frac{C}{N}\le \epsilon
\]
since $||w_i||_A\le ||w_i'||_A$.

Hence

\[
\left| \frac{\langle v^{\pm 1},
w_i'\rangle_{good}}{||w_i'||_A}-\frac{\langle v^{\pm 1},
w_i'\rangle}{||w_i'||_A}\right|\le \epsilon.
\]

Note that by construction $\langle v, w_i'\rangle_{good}=\langle v,
w_i\rangle_{good}$.

Hence by $(\dag)$

\begin{gather*}
\frac{\langle v^{\pm 1}, w_i\rangle_{good}}{||w_i||_A}\ge
\frac{\langle v^{\pm 1},
w_i'\rangle_{good}}{||w_i'||_A}=\frac{\langle v^{\pm 1},
w_i\rangle_{good}}{||w_i||_A}\frac{||w_i||_A}{||w_i'||_A}\ge
\frac{\langle v^{\pm 1}, w_i\rangle_{good}}{||w_i||_A}\frac{N}{N+C}
\end{gather*}

Hence by the choice of $N$ we have
\[
\left|  \frac{\langle v^{\pm 1},
w_i\rangle_{good}}{||w_i||_A}-\frac{\langle v^{\pm 1},
w_i'\rangle_{good}}{||w_i'||_A}\right|\le \epsilon.
\]

Therefore
\[
\left|  \frac{\langle v^{\pm 1},
w_i'\rangle}{||w_i'||_A}-\frac{\langle v^{\pm 1},
w_i\rangle}{||w_i||_A}\right|\le 3\epsilon
\]
and hence
\[
\left|  \frac{\langle v^{\pm 1}, w_i'\rangle}{||w_i'||_A}-\langle v,
\mu\rangle\right|\le 4\epsilon.
\]

Since $w_i'$
represents an element conjugate in $F$ to an element of $H$, we have
verified that for any $\epsilon>0$ and any integer $M\ge 1$ condition
$(\ast\ast)$ holds. This implies the statement of the proposition.
\end{proof}

\begin{prop}\label{prop:restr}
Let $H\le F$ be a nonabelian finitely generated subgroup. Let $T$ be
an $\mathbb R$-tree with a very small isometric action of $F$ such
that $H$ acts nontrivially on $T$. Let $T_H$ be the smallest
$H$-invariant subtree.

Let $\mu\in Curr(F)$ be such that $supp(\mu)\subseteq
i_\Lambda(\partial^2 H)$. Then:

\begin{enumerate}

\item We have
\[
\langle T,\mu\rangle\le \langle T_H,\mu|_H\rangle\le n_{H,F}d_{H,F}\
\langle T,\mu\rangle.
\]

\item Moreover, if $H\le F$ is malnormal, then
\[
\langle T,\mu\rangle= \langle T_H,\mu|_H\rangle.
\]
\end{enumerate}
\end{prop}

\begin{proof}

By Proposition~\ref{prop:appr}, there exist a sequence $h_i\in H$
and a sequence $\lambda_i\ge 0$ such that
\[
\mu=\lim_{i\to\infty} \lambda_i \eta^F_{h_i}.
\]
We may assume that each $h_i$ is nontrivial and is not a proper
power in $H$. For each $i$ let $1\le m_i\le n_{H,F}$ and
$h_{i,1},\dots, h_{i,m_i}\in H$ be such that
\[
[h_i]_F\cap H=\cup_{j=1}^{m_i} [h_{i,j}]_H.
\]
Let $1\le d_i\le d_{H,F}$ be such that in $F$ we have
$h_i=f_i^{d_i}$ where $f_i$ is not a proper power in $F$.

Then by Lemma~\ref{lem:split} $\eta_{h_i}^F|_H=d_i\sum_{j=1}^{m_i}
\eta_{h_{i,j}}^H$.  Hence
\[
\mu|_H=\lim_{i\to\infty} \lambda_i d_i\sum_{j=1}^{m_i}
\eta_{h_{i,j}}^H.
\]
Note that $||h_i||_T=||h_{i,j}||_{T_H}$.

We have
\[
\langle T,\mu\rangle=\lim_{i\to\infty} \lambda_i ||h_i||_T.
\]
Also,

\[
\langle T_H,\mu|_H\rangle=\lim_{i\to\infty} \lambda_i
d_i\sum_{j=1}^{m_i}||h_{i,j}||_{T_H}=\lim_{i\to\infty} \lambda_i
d_im_i||h_i||_{T}\tag{$\ast$}
\]
Since $1\le d_im_i\le n_{H,F}d_{H,F}$, $(\ast)$ implies that
\[
\langle T,\mu\rangle\le \langle T_H,\mu|_H\rangle\le n_{H,F}d_{H,F}\
\langle T,\mu\rangle.
\]
Moreover, if $H\le F$ is malnormal then $n_{H,F}=d_{H,F}=1$ and
hence

\[
\langle T,\mu\rangle= \langle T_H,\mu|_H\rangle,
\]
as required.

\end{proof}

\section{The general case}

In this section we will prove that for an arbitrary very small
action of $F$ on an $\mathbb R$-tree $T$ the condition $\langle
T,\mu\rangle=0$ implies that $supp(\mu)\subseteq L^2(T)$.

We need the following basic fact about the structure of a general very
small action (see Remark~2.1 in \cite{CHL2}):
\begin{propdfn}\label{pd:sim}\cite{CHL2}
Let $F$ be a nonabelian finitely generated group with a very small
minimal isometric action on an $\mathbb R$-tree $T$. Suppose that
this action is neither discrete nor has dense orbits.

Then there exists an $F$-invariant collection $\mathcal V$ of
disjoint closed subtrees of $T$ with the following properties:

\begin{enumerate}
\item The family $\mathcal V$ consists of a finite number of
distinct $F$-orbits of closed subtrees of $T$.
\item For every tree $T'\in \mathcal V$ the set-wise stabilizer
$H=Stab_F(T')$ of $T'$ in $F$ is a finitely generated subgroup of
$F$. Moreover, either $H$ acts on $T'$ with dense orbits or $T'$ is
a single point that is fixed by $H$.
\item Let $T_s$ be an $\mathbb R$-tree obtained from $T$ by collapsing all subtrees from
$\mathcal T$ to points and let $q:T\to T_s$ be the corresponding
$F$-equivariant quotient map. Then the action of $F$ on $T$ factors
through to an isometric action of $F$ on $T_s$ that is very small,
minimal and discrete.
\item The quotient $\mathcal Y=T_s/F$ is a finite reduced graph of
groups where the vertex groups correspond to setwise $F$-stabilizers
of trees from $\mathcal V$ and where the edge-lengths correspond to
distances in $T$ between $F$-orbits of the trees from $\mathcal V$
representing the end-points of the edge in question.
\end{enumerate}

\end{propdfn}

\begin{lem}
Let $q:T\to T_s$ be as in Proposition-Definition~\ref{pd:sim}. Then
the following hold:
\begin{enumerate}
\item For every $f\in F$ and every $p\in T$ we have $d_{T_s}(q(p), f q(p))\le
d_T(p,fp)$.

\item For every $f\in F$ we have $||f||_{T_s}\le ||f||_{T}$.

\item For every $\mu\in Curr(F)$ we have $\langle T_s, \mu\rangle
\le \langle T, \mu\rangle$.

\item We have $L^2(T)\subseteq L^2 (T_s)$.
\end{enumerate}
\end{lem}

\begin{proof}
Part (1) follows from the definition of $q:T\to T_s$. Part (1)
obviously implies part (2). It is also easy to see that part (2)
implies parts (3) and (4).

Indeed, let $\mu\in Curr(F)$. Write $\mu$ as $\mu=\lim_{i\to\infty}
\lambda_i \eta_{g_i}$ for some $g_i\in F, \lambda_i\ge 0$. Then
\[
\langle T_s, \mu\rangle =\lim_{i\to\infty} \lambda_i
||g_i||_{T_s}\le \lim_{i\to\infty} \lambda_i ||g_i||_{T} =\langle T,
\mu\rangle,
\]
so that part (3) holds. To see that part (4) holds, fix a free basis
$A$ of $F$. Suppose $v\in F(A)$ is such that $v\in [L^2(T)]_A$. Then for
any $\epsilon>0$ there is a cyclically reduced $w$ with $||w||_T\le
\epsilon$ and with $v$ being a subword of $w$. Then $||w||_{T_s}\le
||w||_{T}$. Since $\epsilon>0$ was arbitrary, it follows that
$v\in [L^2 (T_s)]_A$, as required.
\end{proof}

\begin{conv}
For the remainder of this section and the subsequent section, let
$F$, $T$, $T_s$ and $Y$ be as in
Proposition-Definition~\ref{pd:sim}. For each vertex $x\in VY$ let
$G_x\le F$ be a subgroup corresponding to the vertex group of $x$ of
$Y$. We will denote by $V'Y$ the set of all those $v\in VY$ such
that the vertex group of $x$ in $\mathcal Y$ is not cyclic.
\end{conv}

\begin{prop}\label{prop:split}
Let $\mu\in Curr(F)$ be such that \[supp(\mu)\subseteq  L^2(T_s)=
\underset{x\in V'Y}{\bigcup} i_\Lambda(\partial^2 G_x).\] Then there
exist currents $\mu_x\in Curr(F), x\in V'Y$, such that
$\mu=\sum_{x\in V'Y} \mu_x$ and such that $supp(\mu_x)\subseteq
i_\Lambda(\partial^2 G_x)$ for every $x\in V'Y$.
\end{prop}

\begin{proof}

There is a finite collection of nontrivial elements
$g_1,\dots, g_p\in F$ such that if for some $x_1,x_2\in V'Y$, $x_1\ne
x_2$ and some $f_1,f_2\in G$ we have $f_1G_{x_1}f_1^{-1}\cap f_2
G_{x_2}f_2^{-1}\ne 1$, then $f_1G_{x_1}f_1^{-1}\cap f_2 G_{x_2}f_2^{-1}$
is a cyclic group generated by a
conjugate of the power of some $g_i$ in $F$. Indeed, suppose that $H=f_1G_{x_1}f_1^{-1}\cap f_2 G_{x_2}f_2^{-1}\ne 1$. Then $H$
fixes a segment in $T_s$ joining a vertex $v_1$ of $T_s$ projecting to $x_1$
with a vertex $v_2$ of $T_s$ projecting to $x_2$. Hence $H$ fixes an
edge of $T_s$ adjacent to $v_1$ . This implies that $H$ is conjugate to a subgroup
of the edge-group for an edge adjacent to $x_1$ in $\mathcal Y$.
Thus we can take $g_1,\dots, g_p\in F$ to be the generators of the
nontrivial edge-groups of $\mathcal Y$.

One can then show that if $x_1,x_2\in V'Y$, $x_1\ne
x_2$ and a point $(\xi,\zeta)\in \partial^2 F$ belongs to
$i_\Lambda(\partial^2 G_{x_1})\cap  i_\Lambda(\partial^2 G_{x_2})$
then for some $1\le i\le p$ the point $(\xi,\zeta)$ is an
$F$-translate of $(g_i^{-\infty}, g_i^{\infty})$ or of $(g_i^{\infty},
g_i^{-\infty})$. To see this, choose a free basis $A$ of $F$ and
suppose that $(\xi,\zeta)\in i_\Lambda(\partial^2 G_{x_1})\cap
i_\Lambda(\partial^2 G_{x_2})$. Then the bi-infinite geodesic joining
$\xi$ to $\zeta$ in the Cayley graph $X(F,A)$ is labelled by a
bi-infinite freely reduced word $w$ that can be read along some
bi-infinite paths in cores of the Stallings subgroup graphs
$\Gamma_{G_{x_1}}$ and $\Gamma_{G_{x_2}}$. Therefore (see \cite{KM})
$w$ can be read along a bi-infinite path in a connected component of
the ``Stallings product graph'' $\Gamma_{G_{x_1}}\times
\Gamma_{G_{x_2}}$ (also known as the push-out of $\Gamma_{G_{x_1}}$ and $\Gamma_{G_{x_2}}$). Hence (again see \cite{KM}) the bi-infinite word $w$ is readable along some path in the core of the Stallings subgroup graph for a subgroup of the form $f_1G_{x_1}f_1^{-1}\cap f_2 G_{x_2}f_2^{-1}$ for some $f_1,f_2\in F$. Therefore $w$ is an infinite power of some $g_i$, as required.

For $1\le i\le p$ put $\lambda_i=\mu\left(\{ (g_i^{-\infty}, g_i^{\infty})
\}\right)=\mu\left(\{ (g_i^{\infty}, g_i^{-\infty})\}\right)$. Thus $\lambda_i\ge 0$.

Hence we can represent $\mu$ as
\[
\mu=\nu +\sum_{i=1}^p \lambda_i \eta_{g_i}
\]
where $\nu\in Curr(F)$ is a geodesic current,  and where for every
$1\le i\le p$ the current $\nu$ has no atom at $(g_i^{-\infty},
g_i^{\infty})$, and hence it has no atom at every $F$-translate of
$(g_i^{-\infty}, g_i^{\infty})$ or of $(g_i^{\infty},
g_i^{-\infty})$. Here by saying that a measure has no atom at a
particular point we mean that the measure of a singleton consisting of
that point is equal to zero.

For the current $\nu$ the statement of the proposition is
obvious. Indeed, for $S\subseteq \partial ^2 F$ and for $x\in V'Y$ put
\[
\nu_x(S)=\nu\big(S\cap i_\Lambda(\partial^2 G_x)\big).
\]
It is not hard to check that for every $x\in V'Y$ we have
$\nu_x\in Curr(F)$ and $supp[\nu_x]\subseteq  i_\Lambda(\partial^2 G_x)$, and that
$\nu=\sum_{x\in V'Y}\nu_x$.

Note that for every $1\le i\le p$ there is some (not necessarily unique) $x(i)=x\in V'Y$ such that
$supp(\eta_{g_i})\subseteq i_\Lambda(\partial^2 G_x)$.

For every $x\in V'Y$ put $\mu_x$ to be the sum of $\nu_x$ and all
those $\lambda_i\eta_{g_i}$ for which $x(i)=x$. Then
$\mu=\sum_{x\in V'Y} \mu_x$ and  $supp(\mu_x)\subseteq
i_\Lambda(\partial^2 G_x)$ for every $x\in V'Y$, as required.
\end{proof}
Note that the decmposition $\mu=\sum_{x\in V'Y} \mu_x$ in
Proposition~\ref{prop:split} is, in general, non-canonical.

\begin{prop}\label{prop:restr1}
Let $H\le F=F(A)$ be a finitely generated subgroup and let $\mu\in
Curr_H(F)$.

Then $supp(\mu)\subseteq i_\Lambda (supp (\mu|_H))$.
\end{prop}

\begin{proof}

Let $\Gamma_H$ be the Stallings subgroup graph of $H$ with respect
to $A$ and let $\Delta_H=Core(\Gamma_H)$. Note that the minimal
$H$-invariant subtree $X_H\subseteq X(F,A)$ is a copy of $\widetilde
{\Delta_H}$ in $X(F,A)$.

By conjugating $H$ if necessary, without loss of generality we may
assume that in fact $\Delta_H=\Gamma_H$, so that the base-vertex
$y\in V\Gamma_H$ is a vertex of $\Delta_H$. Then $(\Gamma_H,y)$
provides a canonical isomorphism (simplicial chart) $\alpha: H\to
\pi_1(\Gamma_H,y)$. This isomorphism in turn provides a canonical
$H$-equivariant identification of $\partial H$ with $\partial
\widetilde{\Gamma_H}=\partial X_H$ and a corresponding
$H$-equivariant identification of $\partial^2 H$ with $\partial^2
X_H$. From this point on we will assume that these identifications
are made without additional comment.

Recall that the lamination $L:=i_\Lambda(supp(\mu|_H))\in
\Lambda^2(F)$ has the form

\[
L=\overline{\cup_{f\in F} f\ supp(\mu|_H)}.
\]

Let $v\in F(A)$ be such that $v\in supp_A(\mu)$, so that $\langle v,
\mu\rangle_A>0$. We need to show that $v\in L_A$.

Fix a segment $[x,y]$ in $X_H$ labelled by $v$.

Since $supp(\mu)\subseteq i_\Lambda(\partial^2 H)$, we know that
\begin{gather*}
\langle v,\mu\rangle=\mu(Cyl_X([x,y]))=\mu\left( Cyl_X([x,y])\cap
\big( \cup_{f\in F} f\partial^2H \big) \right)=\\
=\bigcup_{f\in F} \left(Cyl_X([x,y])\cap f\partial^2 H\right)>0.
\end{gather*}

Since $\mu$ is countably additive, there is some $f\in F$ such that
\[
\mu\left( Cyl_X([x,y])\cap
f\partial^2H\right)=\mu\left(f^{-1}Cyl_X([x,y])\cap
\partial^2H \right)>0.
\]
It is easy to see that if $[x,y]$ is not contained in $fX_H$ then
$Cyl_X([x,y])$ and $f\partial^2H$ are disjoint. Hence we have that
$[x,y]\subseteq fX_H$ and that $[x,y]$ is labelled by $v$ in
$X=X(F,H)$. Thus there is a reduced path $p$ in $\Delta_H$ labelled
by $v$ such that $[x',y']:=f^{-1}[x,y]\subseteq X_H=\widetilde
\Delta_H$ is a lift of $p$ in $\widetilde \Delta_H$. Note that
$f^{-1}Cyl_X([x,y])=Cyl_X([x',y'])$ so that $\mu(Cyl_X([x',y'])\cap
\partial^2 H)>0$. It is easy to see that
\[
Cyl_X([x',y'])\cap
\partial^2H=Cyl_{X_H}([x',y'])\subseteq \partial^2H.
\]
Therefore by definition of $\mu|_H$ we have:
\[
\mu|_H(Cyl_{X_H}([x',y']))=\mu(Cyl_{X_H}([x',y']))=\mu(Cyl_X([x',y'])\cap
\partial^2H)>0.
\]

Thus $\langle p, \mu|_H\rangle_{\alpha}>0$ and hence $p\in
supp_{\alpha}(\mu|_H)$. Since $v\in F(A)$ is the label of $p$, it
follows that $v\in L_A$ where $L=i_\Lambda(supp\ \mu|_H)$, as
required.
\end{proof}

\begin{prop}\label{prop:comp}
Let $x\in V'Y$ and let $\mu\in Curr(F)$ be such that
$supp(\mu)\subseteq i_\Lambda(\partial^2 G_x)$ and $\langle T,
\mu\rangle=0$. Then $supp(\mu)\subseteq L^2(T)$.
\end{prop}

\begin{proof}
Denote $H=G_x$. Fix a free basis $A$ of $F$ and a free basis $B$ of
$H$. If $H$ fixes a vertex of $T$ then $i_\Lambda(\partial^2
H)\subseteq L^2(T)$ and the statement is obvious. Suppose that $H$
acts nontrivially on $T$ and let $T_H$ be the minimal $H$-invariant
subtree of $T$. Then $H$ acts on $T_H$ with dense orbits.

Since $supp(\mu)\subseteq i_\Lambda(\partial^2 H)$ and $\langle T,
\mu\rangle=0$, Proposition~\ref{prop:restr} implies that $\langle
T_H,\mu|_H\rangle=0$. Since $H$ acts on $T_H$ with dense orbits, we
know by Proposition~\ref{prop:dense1} that $supp(\mu|_H)\subseteq
L^2(T_H)\subseteq
\partial^2H$.

Since $supp(\mu)\subseteq i_\Lambda(\partial^2 H)$,
Proposition~\ref{prop:restr1} implies that $supp(\mu)\subseteq
i_\Lambda (supp (\mu|_H))$ and hence
\[
supp(\mu)\subseteq i_\Lambda(L^2(T_H))
\]

It is easy to see that $i_\Lambda(L^2(T_H))\subseteq L^2(T)$ and the
statement of the proposition follows.

\end{proof}

\begin{thm}\label{thm:main1}
Let $F$ be a finitely generated nonabelian free group with a very
small isometric minimal action on an $\mathbb R$-tree $T$. Let
$\mu\in Curr(F)$ be such that $\langle T,\mu\rangle=0$. Then
$supp(\mu)\subseteq L^2(T)$.
\end{thm}
\begin{proof}
If $F$ acts on $T$ with dense orbits or if the action is simplicial,
the statement of the theorem follows from
Proposition~\ref{prop:dense1} and Proposition~\ref{prop:discr1}.
Suppose neither of these two situations occurs. Let $q:T\to
T_s=\widetilde {\mathcal Y}$ be as in
Proposition-Definition~\ref{pd:sim}.

By Proposition~\ref{prop:split} we can decompose $\mu$ as
\[
\mu=\sum_{x\in V'Y} \mu_x,
\]
where each $\mu_x\in Curr(F)$ satisfies $supp(\mu_x)\subseteq
i_\Lambda(\partial^2 H)$.

We have
\[
0=\langle T,\mu\rangle=\sum_{x\in V'Y} \langle T,\mu_x\rangle
\]
and hence $\langle T,\mu_x\rangle=0$ for each $x\in V'Y$.

Therefore by Proposition~\ref{prop:comp} we have
$supp(\mu_x)\subseteq L^2 T$ for every $x\in VX$.

Since, obviously, $supp(\mu)=\cup_{x\in V'Y} supp(\mu_x)$, it
follows that $supp(\mu)\subseteq L^2(T)$, as required.

\end{proof}

\begin{lem}\label{lem:lam}
Let $F=F(A)$ be a finitely generated nonabelian free group. Let $T$
be an $\mathbb R$-tree with a very small nontrivial minimal
isometric action of $F$ on $T$. Then $L^2(T)\ne \partial^2 F$.
\end{lem}
\begin{proof}
Put $L=L^2(T)$. We need to show that there exists $v\in F(A)$ such
that $v\not\in L_A$.

Let $p\in T$. By Proposition~\ref{prop:BBT}, there is $C<\infty$
such that $BBT_{T,p}(A)<C$.

Since the action of $F$ on $T$ is nontrivial, there exists a freely
reduced $v\in F(A)$ such that $||v||_T\ge 4C+2$. Hence $d_T(p,vp)\ge
||v||_T\ge 4C+2$. Suppose that $v\in L_A$. Then there exists a
cyclically reduced word $w$ in $F(A)$ such that $v$ is an initial
segment of $w$ and such that $||w||_T\le 1$. We can write $w$ as a
reduced product $w=vu$. by Lemma~\ref{lem:LL} we have
\[
d_T(p,wp)\ge d_T(p,vp)+d_T(p,up)-2C\ge ||v||_T-2C.
\]

Since $w$ is cyclically reduced, Lemma~\ref{lem:LL} also implies
that
\[
||w||_T\ge d_T(p,wp)-2C\ge ||v||_T-4C\ge 2.
\]
This contradicts our assumption that $||w||_T\le 1$. Thus $v\not\in
L_A$.

\end{proof}

\begin{cor}\label{cor:full}
Let $\mu\in Curr(F)$ be a current with full support. Then for every
very small action of $F$ on an $\mathbb R$-tree $T$ we have $\langle
T, \mu\rangle>0$.
\end{cor}

\begin{proof}
Suppose there is some very small $T$ such that $\langle \mu,
T\rangle=0$. By Theorem~\ref{thm:main1}  it follows that
$supp(\mu)\subseteq L^2(T)$. By Lemma~\ref{lem:lam} this contradicts
our assumption that $\mu$ has full support.
\end{proof}

\section{Length compactness for currents with full support}

\begin{defn}[Automorphic length spectrum of a current]

  Let $T\in \overline{cv}(F)$ and $\mu\in Curr(F)$. The
  \emph{automorphic length spectrum of $\mu$ with respect to $T$ } is
  the set
\[
{\mathcal S}_T(\mu):=\{ \langle T, \phi \mu\rangle: \phi\in Out(F)\}\subseteq
\mathbb R.
\]
\end{defn}

Note that by $Out(F)$-invariance of the intersection form, we always have
$\langle T, \phi \mu\rangle=\langle \phi^{-1}T, \mu\rangle$. Hence
\[
{\mathcal S}_T(\mu):=\{ \langle \phi T, \mu\rangle: \phi\in Out(F)\}.
\]

\begin{thm}\label{thm:comp}
Let $\mu\in Curr(F)$ be a current with full support and let $T\in
cv(F)$. Then:

\begin{enumerate}
\item For any $C>0$ the set
\[
\{\phi\in Out(F): \langle T, \phi \mu\rangle \le C\}
\]
is finite.

\item The set ${\mathcal S}_T(\mu)$ is a discrete subset of $\mathbb R_{\ge 0}$.

\item Suppose $\phi_n\in Out(F)$ is an infinite sequence of distinct
elements such that for some $\lambda_n\ge 0$ and some $\mu'\in
Curr(F)$ we have $\lim_{n\to\infty} \lambda_n \phi_n \mu =\mu'$.
Then $\lim_{n\to\infty} \lambda_n=0$.
\end{enumerate}

\end{thm}
\begin{proof}
It is obvious that (1) implies (2).

To see that (1) holds, suppose that for some $C>0$ there exists an
infinite sequence of distinct elements $\phi_n\in Out(F)$ such that
for every $n\ge 1$ we have $\langle T, \phi_n \mu\rangle \le C$.
By rescaling $T$, we may assume that the graph $T/F$ has volume $1$,
that is $T\in CV(F)$.

Since $\overline{CV}(F)=CV(F)\cup
\partial CV(F)$ is compact, after passing to a subsequence we may
assume that $\lim_{n\to\infty} [\phi_n^{-1} T]=[T_\infty]$ in
$\overline{CV}(F)$ for some $T_\infty\in \overline{cv}(F)$. Thus there is a
sequence $c_n\ge 0$ such that
\[
\lim_{n\to\infty} c_n \phi_n^{-1} T=T_{\infty}\tag{$\clubsuit$}
\]
in $\overline{cv}(F)$. Moreover, since all $\phi_n\in Out(F), n\ge
1$ are distinct and the action of $Out(F)$ on $CV(F)={\mathbb P}
cv(F)$ is properly discontinuous, we have $[T_{\infty}]\in \partial
CV(F)$. This implies that $\lim_{n\to\infty} c_n=0$.

Indeed, suppose not. Then, after passing to a subsequence, we may
assume that $c_n\ge c>0$ for every $n\ge 1$. Since
$T_\infty\not\in cv(F)$, there are nontrivial elements in $F$ acting
on $T_\infty$ with arbitrary small translation length. That is,
there exists a sequence $g_i\in F, g_i\ne 1$ such that
$\lim_{i\to\infty} ||g_i||_{T_\infty}=0$. Recall that by definition of the left action of
$Out(F)$ on $\overline{cv}(F)$, we have $||g||_{\phi_{n}^{-1} T}=||\phi_{n}(g)||_T$ for every $g\in F$ and every $n\ge 1$.
Then $(\clubsuit)$ implies
that there is a sequence $n_i\ge 1$ with $\lim_{i\to\infty}
n_i=\infty$ such that
\[
0=\lim_{i\to\infty} c_{n_i} ||g_i||_{\phi_{n_i}^{-1} T}=\lim_{i\to\infty} c_{n_i} ||\phi_{n_i}g_{i}||_T.
\]
On the other hand, since the action of $F$ on $T$ is free and
simplicial, there is some $\delta>0$ such that for every $f\in F, f\ne
1$ we have $||f||_T\ge \delta$. Hence for every $i\ge 1$
\[
c_{n_i} ||\phi_{n_i}g_{i}||_T\ge c\delta>0,
\]
yielding a contradiction. Thus indeed $\lim_{n\to\infty}
c_n=0$.

By $Out(F)$-invariance of the intersection form we have
\[
0\le \langle c_n
\phi^{-1}_nT,\mu\rangle=c_n\langle
\phi^{-1}_nT,\mu\rangle=c_n \langle T, \phi_n \mu\rangle \le
c_n C \to_{n\to \infty} 0.
\]
Therefore, by the continuity of the intersection form~\cite{KL4}, we have
\[
0=\lim_{n\to\infty}\langle c_n
\phi^{-1}_nT,\mu\rangle=\langle T_{\infty},\mu\rangle.
\]
However, this contradicts the conclusion of Corollary~\ref{cor:full}
since by assumption $\mu$ has full support. Thus part (1) is
established.

We now show that (1) implies (3). Let $\lambda_n,\phi_n,\mu'$ be as
in (3). Part (1) implies that $\lim_{n\to\infty} \langle T, \phi_n
\mu\rangle=\infty$. On the other hand,
\[
\lim_{n\to\infty} \lambda_n \langle T, \phi_n \mu\rangle
=\lim_{n\to\infty} \langle T, \lambda_n\phi_n \mu\rangle =\langle T,
\mu'\rangle<\infty.
\]
This implies that $\lim_{n\to\infty} \lambda_n=0$, as required.
\end{proof}

\section{Unique ergodicity}\label{sect:ue}

Recall that an element $\phi\in Out(F)$ is called \emph{reducible}
if there exists a free product decomposition $F=C_1\ast\dots C_k\ast
F'$, where $k\ge 1$ and $C_i\ne \{1\}$, such that $\phi$ permutes the
conjugacy classes of subgroups $C_1,\dots, C_k$ in $F$. An element
$\phi\in Out(F)$ is called \emph{irreducible} if it is not
reducible.

\begin{defn}
\label{defniwip}
An element $\phi\in Out(F)$ is said to be
\emph{irreducible with irreducible powers} or an \emph{iwip} for
short,  if for every $n\ge 1$ $\phi^n$ is irreducible (sometimes such
automorphisms are also called \emph{fully irreducible}). Thus $\phi\in
Aut(F)$ is an iwip if and only if  no positive power of $\phi$
preserves the conjugacy class of a proper free factor of $F$. An
element $\phi\in Aut(F)$ is \emph{atoroidal} that is, if there does not exist a nontrivial conjugacy class in $F$ that is fixed by some positive power of $\phi$.
\end{defn}

Let $\phi\in Out(F)$ be an atoroidal iwip. It is known, by the work of Reiner Martin in the case of currents and by the result of Levitt and Lustig in the case of $\overline{CV}(F)$ that the (left) action of $\phi$ has ``North-South'' dynamics on both $\mathbb PCurr(F)$ and $\overline{CV}(F)$:
\begin{prop}\label{prop:ns}
Let $\phi\in Out(F)$ be an atoroidal iwip.
Then the following hold:
\begin{enumerate}
\item \cite{LL} The action of $\phi$ on $\overline{CV}(F)$ has exactly two distinct fixed points $[T_+],[T_-]$ and, moreover, for any $[T]\in \overline{CV}(F)$, $[T]\ne [T_-]$ we have $\lim_{n\to\infty}\phi^n[T]=[T_+]$ and for any $[T]\in \overline{CV}(F)$, $[T]\ne [T_+]$ we have $\lim_{n\to\infty}\phi^{-n}[T]=[T_-]$.
\item \cite{Ma} The action of $\phi$ on $\mathbb PCurr(F)$ has exactly two distinct fixed points $[\mu_+],[\mu_-]$ and, moreover, for any $[\mu]\in \mathbb PCurr(F)$, $[\mu]\ne [\mu_-]$ we have $\lim_{n\to\infty}\phi^n[\mu]=[\mu_+]$ and for any $[\mu]\in \mathbb PCurr(F)$, $[\mu]\ne [\mu_+]$ we have $\lim_{n\to\infty}\phi^{-n}[\mu]=[\mu_-]$.

\item \cite{KL4} We have $\langle T_+,\mu_+\rangle=\langle T_-,\mu_-\rangle=0$ and $\langle T_+,\mu_-\rangle > 0, \langle T_-,\mu_-\rangle>0.$
\end{enumerate}
\end{prop}

Here we prove that $\mu_+$ is uniquely ergodic in the following sense:

\begin{thm}\label{thm:ue1}
Let $\phi\in Out(F)$, $\mu_{\pm}$, $T_{\pm }$ be as in Proposition~\ref{prop:ns}. If $[\mu]\in \mathbb PCurr(F)$ is such that $supp(\mu)\subseteq supp(\mu_+)$ then $[\mu]=[\mu_+]$.
\end{thm}
\begin{proof}
Note that since $\phi$ fixes $[T_+]$, we have $\phi T_+=\lambda T_+$ for some $\lambda>0$.
Suppose that $supp(\mu)\subseteq supp(\mu_+)$ but $[\mu]\ne [\mu_+]$. Since $\langle T_+,\mu_+\rangle=0$, Theorem~\ref{thm:main} implies that $supp(\mu_+)\subseteq L^2(T_+)$. Since $supp(\mu)\subseteq supp(\mu_+)\subseteq L^2(T_+)$, Theorem~\ref{thm:main} also implies that $\langle T_+,\mu\rangle=0$. Since $[\mu]\ne [\mu_+]$ we have $\lim_{n\to\infty}\phi^{-n}[\mu]=[\mu_-]$, so that for some sequence $c_n>0$ we have $\lim_{n\to\infty} c_n\phi^{-n}\mu=\mu_-$ in $Curr(F)$. Hence for any $n\ge 1$
\[
\langle T_+, c_n\phi^{-n}\mu\rangle=\langle \phi^n T_+, c_n\mu\rangle=\langle \lambda^n T_+,c_n\mu\rangle=\lambda^nc_n\langle T_+,\mu\rangle=0.
\]
Therefore, by continuity of the intersection form
\[
0=\lim_{n\to\infty} \langle T_+, c_n\phi^{-n}\mu\rangle=\langle T_+, \lim_{n\to\infty} c_n\phi^{-n}\mu\rangle=\langle T_+,\mu_-\rangle,
\]
which contradicts part (3) of Proposition~\ref{prop:ns}.
\end{proof}

By a similar argument we obtain a dual statement for $T_+$.

\begin{thm}\label{thm:ue2}
Let $\phi\in Out(F)$, $\mu_{\pm}$, $T_{\pm }$ be as in Proposition~\ref{prop:ns}. Let $[T]\in \overline{CV}(F)$ be such that $L^2(T_+)\subseteq L^2(T)$. Then $[T]=[T_+]$.
\end{thm}
\begin{proof}
Note that since $\phi$ fixes $[\mu_+]$, there is $r>0$ such that $\phi \mu_+=r \mu_+$.
Suppose that $[T]\in \overline{CV}(F)$ be such that $L^2(T_+)\subseteq L^2(T)$ but that $[T]\ne [T_+]$.
We have $supp(\mu_+)\subseteq L^2(T_+)\subseteq L^2(T)$ and therefore by Theorem~\ref{thm:main} we have $\langle T,\mu_+\rangle=0$. Since $[T]\ne [T_+]$, we have $\lim_{n\to\infty}\phi^{-n}[T]=[T_-]$, so that for some sequence $c_n>0$ we have $\lim_{n\to\infty} c_n\phi^{-n}T=T_-$ in $\overline{cv}(F)$. Hence for any $n\ge 1$
\[
\langle c_n\phi^{-n}T, \mu_+\rangle=\langle c_nT, \phi^n\mu_+\rangle=\langle c_nT, r^n \mu_+\rangle=r^nc_n\langle T,\mu_+\rangle=0.
\]
Therefore, by continuity of the intersection form
\[
0=\lim_{n\to\infty} \langle c_n\phi^{-n}T,\mu_+\rangle=\langle \lim_{n\to\infty} c_n\phi^{-n}T, \mu_+\rangle=\langle T_-,\mu_+\rangle,
\]
which contradicts part (3) of Proposition~\ref{prop:ns}.
\end{proof}

\begin{cor}
Let $\phi\in Out(F)$, $\mu_{\pm}$, $T_{\pm }$ be as in Proposition~\ref{prop:ns}.
\begin{enumerate}
\item Let $T\in \overline{cv}(F)$. Then $\langle T,\mu_+\rangle=0$ if and only if $[T]=[T_+]$.
\item Let $\mu\in Curr(F)$, $\mu\ne 0$. Then $\langle T_+,\mu\rangle=0$ if and only if $[\mu]=[\mu_+]$.
\end{enumerate}
\end{cor}
\begin{proof}
(1) We already know that $\langle T_+,\mu_+\rangle=0$ and that $supp(\mu_+)\subseteq L^2(T_+)$. Suppose $\langle T,\mu_+\rangle=0$ for some $T\in\overline{cv}(F)$. Suppose that $[T]\ne [T_+]$. Then $\lim_{n\to\infty} \phi^{-n} [T]=[T_-]$, so that $\lim_{n\to\infty} c_n\phi^{-n} T=T_-$ for some $c_n>0$.
We have
\[
\langle c_n\phi^{-n}T,\mu_+\rangle=\langle c_n T,\phi^n\mu_+\rangle=\langle c_n T, r^n\mu_+\rangle=c_nr^n\langle T,\mu_+\rangle=0.
\]
Since $\lim_{n\to\infty} c_n\phi^{-n} T=T_-$, the continuity of the intersection form implies that $\langle T_-,\mu_+\rangle=0$, yielding a contradiction with part (3) of Proposition~\ref{prop:ns}.

The proof of part (2) is essentially symmetric and we omit the details.
\end{proof}

As noted in the above argument, we do know that $supp(\mu_+)\subseteq
L^2(T_+)$, but even in this particular case it can happen that the inclusion is a strict one.

One can check, by directly comparing the
definitions, that, under the assumptions of Proposition~\ref{prop:ns},
$supp(\mu_+)$ is equal to the ``stable lamination'' of $\phi$ in the
sense of \cite{BFH97}.
The reason for the potential inequality $supp(\mu_+)\neq
L^2(T_+)$ comes from the possibility
that $(\xi,\zeta),(\zeta,\omega)\in supp(\mu_+)$
but $(\xi,\omega)\not\in supp(\mu_+)$, while this type of behavior is
by definition
impossible in $L^2(T_+)$.

We believe that that for any atoroidal iwip automorphism
$L^2(T_+)$ is obtained from the stable lamination of $\phi$ via the operation of ``diagonal
closure'' as indicated above.

\section{Filling elements, filling currents and bounded translation
equivalence.}

\begin{defn}\label{defn:fill}
Let $\mu\in Curr(F)$. We say that $\mu$ \emph{fills} $F$ if for
every very small action of $F$ on an $\mathbb R$-tree $T$ we have
$\langle T,\mu\rangle >0$.

Similarly, we say that an element $g\in F$ \emph{fills} $F$ if for
every very small action of $F$ on an $\mathbb R$-tree $T$ we have
$||g||_T>0$. Thus $g$ fills $F$ if and only if $\eta_g$ fills $F$.
\end{defn}

Corollary~\ref{cor:full} says that every current $\mu\in Curr(F)$
with full support fills $F$.

\begin{prop}\label{prop:fill}
Let $\mu\in Curr(F)$ be a current with full support. Let $\mu_n\in
Curr(F)$ be a sequence such that  $\lim_{n\to\infty} \mu_n
 =\mu$. Then there is $n_0\ge 1$ such that for every $n\ge
n_0$ the current $\mu_n$ fills $F$.
\end{prop}

\begin{proof}
Suppose the statement of the proposition fails. Then there exists a
sequence $n_i$ with $\lim_{i\to\infty} n_i=\infty$ and a sequence of
very small $\mathbb R$-trees $T_i$ such that $\langle
T_i,\mu_{n_i}\rangle=0$. Since $CV(F)\cup \partial CV(F)$ is
compact, there exists a sequence $r_i\ge 0$ and a very small action
of $F$ on an $\mathbb R$-tree $T$ such that $\lim_{i\to\infty} r_i
T_i= T$. Note that we have
\[\langle r_i T_i,
\mu_{n_i}\rangle=r_i\langle T_i,
 \mu_{n_i}\rangle=0.\]
 By the continuity of the intersection form on the closure of the
 non-projectivized Outer space (see~\cite{KL4}), this implies that
 \[
\langle T, \mu\rangle=0.
 \]
This contradicts the fact that, by Corollary~\ref{cor:full}, $\mu$
fills $F$.
\end{proof}

Proposition~\ref{prop:fill} immediately implies:

\begin{cor}\label{cor:fill}
Let $\mu\in Curr(F)$ be a current with full support. Let
$\lambda_n\ge 0$ and $g_n\in F$ be such that $\lim_{n\to\infty}
\lambda_n \eta_{g_n} =\mu$. Then there is $n_0\ge 1$ such that for
every $n\ge n_0$ the element $g_n$ fills $F$.
\end{cor}

\begin{notation}
Let $F$ be a finitely generated free group and let $A$ be a free
basis of $F$. Let $\xi\in \partial F$. We represent $\xi$ by a
right-infinite freely reduced word $x_1x_2\dots x_n \dots$, where
$x_i\in A^{\pm 1}$, labelling the geodesic ray from $1$ to $\xi$ in
the Cayley graph $X(F,A)$. For every $n\ge 1$ we denote by
$\xi_A(n)$ the element of $F$ represented by the initial segment of
this ray of length $n$, that is $\xi_A(n)=x_1\dots x_n\in F$.
\end{notation}

\begin{defn}[Uniform measure corresponding to a free basis]\label{defn:ma}
Let $A$ be a free basis of $F$ and let $k\ge 2$ be the rank of $F$.
For a nontrivial freely reduced word $v\in F(A)$ let $Cyl_A(v)$ be
the set of all $\xi\in \partial F$ such that $v$ is an initial
segment of $\xi$, when $\xi$ is realized as a geodesic ray with
origin $1\in F$ in the Cayley graph $X(F,A)$.

The \emph{uniform measure on $\partial F$, corresponding to $A$},
denoted $\mu_A$, is a Borel probability measure on $\partial F$,
such that for every nontrivial freely reduced word $v\in F(A)$ we
have
\[
\mu_A\left( Cyl_A(v) \right)=\frac{1}{2k(2k-1)^{n-1}},
\]
where $n=|v|_A$.
\end{defn}

Informally, a $\mu_A$-random point $\xi\in \partial F$ corresponds
to a ``random'' right-infinite freely reduced word over $A^{\pm 1}$.
We refer the reader to \cite{Ka2,KKS,KN} for a more detailed
discussion regarding the uniform measure $\mu_A$ and the uniform
current $\nu_A$.

\begin{thm}\label{thm:fill}
Let $F=F(A)$ be a finitely generated nonabelian free group. Let
$\mu_A$ be the uniform measure on $\partial F$ corresponding to $A$.
Then there exists a set $R\subseteq \partial F$ with the following
properties:

\begin{enumerate}
\item We have $\mu_A(R)=1$.
\item For each $\xi\in R$ there is $N\ge 1$ such that for every
$n\ge N$ the element $\xi_A(n)\in F$ fills $F$.
\end{enumerate}
\end{thm}

\begin{proof}
Let $\nu_A$ be the uniform current on $F$ corresponding to $A$. Note
that by construction $\nu_A$ has full support. As shown in
\cite{Ka2}, there exists a subset $R\subseteq
\partial F$ with $\mu_A(R)=1$ such that for each $\xi\in R$ we have
\[
\lim_{n\to\infty} \frac{\eta_{\xi_A(n)}}{n}=\nu_A.
\]

Corollary~\ref{cor:fill} implies that for every $\xi\in R$ there is
$N\ge 1$ such that for each $n\ge N$ the element $\xi_A(n)$ fills
$F$, as required.
\end{proof}

Theorem~\ref{thm:fill} says that an ``almost generic'' element of
$F(A)$ fills $F$.

\begin{defn}
We say that nontrivial elements $g,h\in F$ are \emph{boundedly
translation equivalent} in $F$, denoted $g\equiv_b h$, if there is
$C\ge 1$ such that for every free and discrete action of $F$ on an
$\mathbb R$-tree $T$ we have

\[
\frac{1}{C} ||h||_T\le ||g||_T\le C ||h||_T
\]

\end{defn}
Note that in the above definition we can replace ``every free and
discrete action'' by ``every very small action''.

\begin{prop}\label{prop:bta}
Let $f,g\in F$ both fill $F$. Then $f\equiv_b g$ in $F$.
\end{prop}
\begin{proof}

Consider the following function $D: CV(F)\cup \partial CV(F)\to
\mathbb R$. For every very small action of $F$ on an $\mathbb
R$-tree $T$ put
\[
D([T]):=\frac{||f||_T}{||g||_T}
\]
Since $f$ and $g$ both fill $F$, the function $D$ is well-defined on
$CV(F)\cup \partial CV(F)$ and, moreover, $D([T])>0$ for every
$[T]\in CV(F)\cup \partial CV(F)$. It is also clear that $D$ is
continuous. Since $CV(F)\cup \partial CV(F)$ is compact, it follows
that $D$ achieves a positive minimum and a positive maximum on
$CV(F)\cup \partial CV(F)$. This implies that $f\equiv_b g$, as
claimed.
\end{proof}

Theorem~\ref{thm:fill} and Proposition~\ref{prop:bta} show that the
phenomenon of bounded translation equivalence is ``almost generic'' in
$F(A)$.

\end{document}